\documentclass[11pt]{article}
\usepackage{latexsym,amssymb,amsmath,amsfonts,amsthm}
\usepackage{graphics}
\usepackage{graphicx}
\usepackage{mathrsfs}
\usepackage{subfigure}
\usepackage[titletoc,page]{appendix}
\newcommand{\mat}{\mathbb}
\newcommand{\R}{{\mat R}}

\newcommand{\N}{{\mat N}}
\newcommand{\C}{{\mat C}}

\newcommand{\ds}{\displaystyle}
\newcommand{\no}{\nonumber}
\newcommand{\ba}{\backslash}
\newcommand{\pa}{\partial}

\newcommand{\ov}{\overline}

\newcommand{\dive}{{\rm div}}

\newcommand{\Om}{\Omega}
\newcommand{\om}{\omega}
\newcommand{\sig}{\sigma}

\newcommand{\na}{\nabla}

\newcommand{\la}{\lambda}

\newcommand{\De}{\Delta}
\newcommand{\wid}{\widetilde}
\newcommand{\bs}{\boldsymbol}

\newcommand{\PML}{{\rm PML}}
\newtheorem{theorem}{Theorem}[section]
\newtheorem{lemma}[theorem]{Lemma}

\newtheorem{remark}[theorem]{Remark}

\providecommand{\Rp}{\operatorname{Re}}
\providecommand{\Ip}{\operatorname{Im}}
\providecommand{\Diag}{\operatorname{diag}}
\begin{document}
\renewcommand{\theequation}{\arabic{section}.\arabic{equation}}

\title{\bf Convergence of the PML method for scattering problems in poroelastic media}
\author{Qianyuan Yin\thanks{Academy of Mathematics and Systems Science, Chinese Academy of Sciences,
Beijing 100190, China and School of Mathematical Sciences, University of Chinese Academy of Sciences,
Beijing 100049, China ({\tt yinqianyuan@amss.ac.cn})}
\and
Changkun Wei\thanks{The corresponding author, Beijing Key Laboratory of Biological Big Data and Topological Statistics, 
School of Mathematics and Statistics, Beijing Jiaotong University,
Beijing 100044, China ({\tt ckwei@bjtu.edu.cn})}
\and
Bo Zhang\thanks{SKLMS and Academy of Mathematics and Systems Science, Chinese Academy of Sciences,
Beijing 100190, China and School of Mathematical Sciences, University of Chinese Academy of Sciences,
Beijing 100049, China ({\tt b.zhang@amt.ac.cn})}
}
\date{}

\maketitle

\begin{abstract}
This paper is concerned with the time-harmonic wave scattering problems in three dimensional 
poroelastic media. By introducing an intermediate variable $p$, the original $\mathbf{u}-\mathbf{w}$ 
system is equivalently transformed into a $\mathbf{u}-p$ system with fewer degrees of freedom, which 
facilitates the derivation of the fundamental solution, Green's identity and positivity of the 
complex wave numbers. A perfectly matched layer (PML) method is then introduced in the spherical 
coordinates to truncate the unbounded scattering problem. Under certain assumptions on the poroelastic 
and PML parameters, we prove the existence and uniqueness of solutions to the PML problems both in the 
truncated domain and layer. Moreover, the exponential convergence of the PML method is established in 
terms of the thickness and parameters of the PML layer. The proof is based on the PML extension and 
the exponential decay properties of the stretched fundamental solution. As far as we know, this is the 
first convergence result of the PML method for poroelastic scattering problems.

\vspace{.2in}
{\bf Keywords:} Biot, poroelastic wave equations, PML, convergence
\end{abstract}

\section{Introduction}
\setcounter{equation}{0}

Poroelastic wave scattering problem can be effectively modeled by Biot's theory, which extends conventional acoustic
or pure elastic models by incorporating factors such as fluid, porosity, pressure, etc \cite{Biot1956I,Biot1956II}.
Biot's theory of poroelasticity is highly relevant to many fields of applications such as geophysical exploration 
of oil or gas, bio-engineering and materials science \cite{Boer2012,Schanz2003}. The study on this problem has been 
a subject of interest in both mathematical and engineering communities (see \cite{Schanz2009,Yamamoto2004,Zhang2021,
Zhang2025} and the reference therein).

The truncation of the unbounded domain remains a challenging topic for wave scattering problems. Absorbing boundary 
conditions (ABCs) play a crucial role in truncation for both mathematical analysis and numerical computation \cite{HW2013},
among which the perfectly matched layer (PML) method is widely used due to its simplicity, flexibility and effectiveness. 
Originally proposed by Bérenger in 1994 for solving time-dependent Maxwell's equations \cite{Berenger1994}, the PML method 
surrounds the computational domain with a specially designed medium in a finite thickness layer in which the scattered waves 
decay rapidly regardless of the wave incident angle, thereby greatly reducing the computational complexity of the scattering 
problem. A fundamental and challenging issue about PML method is the convergence analysis, i.e., the error estimate in the 
computational domain between the solutions for the original scattering problem and the truncated PML problem, since the 
truncated boundary generates reflected waves which will contaminate the solution in practical computation. 

The convergence analysis of the PML method has been widely studied for the acoustic and electromagnetic wave scattering problems;
see, e.g., \cite{BW2005,Bramble2007,BP2013,CL2005,1998On,wyz2020,wyz2021}. Several works are also available on the convergence of 
the PML method for the elastic wave scattering problems. For the elastic obstacle scattering problem, the exponential convergence
of a spherical PML method was established in \cite{BPT2010}, and the exponential convergence of a Cartesian PML method were 
established in \cite{CXZ2016}. For the elastic scattering problems in periodic structures, the exponential convergence for the
proposed PML method was established in \cite{Jiang2017,Jiang2018} for both one-dimensional and two-dimensional cases. However,
the convergence analysis of the PML methods is quite rare for the relatively complicated interaction scattering problems between 
multi-physical fields. We refer to \cite{JL2017,Zhu2024} for the exponential convergence result of the PML method for the 
acoustic-elastic and electromagnetic-elastic interaction scattering problems. Recently, Yin et al. \cite{ywz2026} established the
well-poseness and exponential convergence of the PML method for the three-dimensional thermoelastic wave scattering problems. 
To the best of our knowledge, there is no convergence
theory of the PML method for poroelastic scattering problems.

In this paper, we focus on the theoretical analysis for three-dimensional poroelastic obstacle scattering problems. 
Biot's initial formulation $\mathbf{u}-\mathbf{w}$ is confined to the displacements of solid and fluid phases. 
The time-harmonic poroelastic wave equations are reformulated by replacing the fluid displacement with pressure $p$, 
thereby reducing the number of degrees of freedom. An abstract $\mathbf{u}-p$ system is thus obtained that incorporates 
thermoelastic and poroelastic wave models into a unified framework, thereby addressing the key issues of the mathematical model,
such as the derivation of the fundamental solution and Green's identity. We prove that the real and imaginary part of the three
wave numbers for the poroelastic system are positive under some reasonable constraints on these poroelatic parameters. This fact
plays a key role in the PML analysis for the poroelastic scattering problems. Finally, we propose a spherical PML method to truncate
the unbounded poroelastic scattering domain into a bounded computational domain. The complex spherical coordinate stretching to
derive the PML problem is \cite{Chew1997}:
\begin{equation*}
\wid{r}(r)=\int_0^{r} s(\tau)d\tau, \; s=1+{\bf i}\sig\;(\sigma\geq 0)
\end{equation*}
where the continuous function $s(\tau)$ is the model medium property. We remark that the PML equation is derived directly from
the original $\mathbf{u}-\mathbf{w}$ formulation to make use of the symmetry of the system. Under some assumptions on the 
poroelastic and PML medium parameters, we prove the existence and uniqueness of the solutions to the PML equation in the truncated 
domain and layer.  We estimate the PML extension of the poroelastic boundary integral representation by the geometric and polynomial
methods. The exponential convergence of the PML method is established in terms of the thickness and parameters of the PML layer.

The paper is organized as follows. In section \ref{sec:scattering model}, we present the mathematical model for the time‑harmonic poroelastic scattering problem, introduce an abstract $\mathbf{u}-p$ system, and derive the fundamental solution and Green’s identity, as well as establish the positivity of the complex wave numbers. Section \ref{sec:PML} is devoted to propose the PML method for the poroelastic scattering problem based on a complex spherical coordinate stretching.  The well-posedness for the poroelastic PML equation in the truncated domain 
and layer are then studied and carefully analyzed, and the exponential convergence result
is finally established for the proposed PML method. Some conclusions are given in section \ref{sec:conclusion}.

\section{Functional spaces}\label{sec:fs}
\setcounter{equation}{0}

In this section, we introduce the functional spaces needed in this paper. Let $D\subset  \R^3$ be a domain with smooth boundary, 
and let $\bs{n}$ be the unit outward normal vector on its boundary $\pa D$. For $m\in\N$, denote the standard scalar Sobolev 
space by 
\begin{equation*}
H^m(D)=\{ \pa_{\alpha}u\in L^2(D)\;{\rm for\;all}\;\lvert \alpha \rvert \leq m\},
\end{equation*}
endowed with the norm 
\begin{equation*}
\Vert u\Vert_{H^m(D)}=\left(\ds\sum_{|\alpha|\leq m}\Vert\pa_{\alpha}u\Vert_{L^2(D)}^2\right)^{\frac{1}{2}}.
\end{equation*}
The space $H_0^{m}(D)$ is the closure of $C_0^{\infty }(D)$ in $H^m(D)$ with respect to the norm $\Vert \cdot 
\Vert _{H^m(D)}$.

Let $ H^1(D)^3 $ be the Cartesian product spaces equipped with the corresponding $2$-norms of $ H^1(D) $. 
For a vector field $\bs{v}=(v_1,v_2,v_3)^{\top}\in H^1(D)^3$, define the gradient tensor
\begin{equation*}
\na\bs{v}=\left[
\begin{matrix}
\pa_{x_1}v_1 & \pa_{x_2}v_1 & \pa_{x_3} v_1 \\
\pa_{x_1}v_2 & \pa_{x_2}v_2 & \pa_{x_3} v_2 \\
\pa_{x_1}v_3 & \pa_{x_2}v_3 & \pa_{x_3} v_3
\end{matrix}\right],
\end{equation*}
with the Frobenius norm
\begin{equation*}
\|\nabla \bs{v}\|_{L^2(D)^{3 \times 3}} := \left( \int_{D} \nabla \bs{v}:\nabla \overline{\bs{v}} dx 
\right)^{1/2}=\left( \sum_{j=1}^3 \int_{D} |\nabla v_j|^2 dx \right)^{1/2},
\end{equation*}
where $A:B={\rm tr}(AB^{\top})$ denotes the Frobenius inner product of the square matrices $A$ and $B$.

For a vector field $\bs{v}=(v_1,v_2,v_3)^{\top}$, we introduce the functional spaces
\begin{equation*}
\begin{aligned}
H(\dive,D) &= \left\{\bs{v}\in L^2(D)^3: \nabla \cdot\bs{v}\in L^2(D)\right\}, \\
H_0(\dive,D) &= \left\{\bs{v}\in H(\dive,D): \bs{n}\cdot \bs{v}=0\;{\rm on}\;\pa D\right\}.
\end{aligned}
\end{equation*}
For a bounded domain $D$, we define the following local spaces by
\begin{equation*}
\begin{aligned}
H_{loc}^1(\R^3\ba\ov{D}) &= \{u:\ \R^3\ba\ov{D}\to\mathbb{C}:u|_B\in H^1(B)\;\text{for all balls}\;B\},\\
H_{loc}(\dive,\R^3\ba\ov{D}) &= \{\bs{v}:\ \R^3\ba\ov{D}\to\mathbb{C}^3:\bs{v}|_B\in H(\dive,B)
\;\text{for all balls}\;B\}.
\end{aligned}
\end{equation*}

\section{Poroelastic scattering problem}\label{sec:scattering model}
\setcounter{equation}{0}

In this section, we present the mathematical formulation of the three-dimensional poroelastic 
scattering problem based on Biot's theory. By introducing an intermediate variable $p$, the 
original $\mathbf{u}-\mathbf{w}$ system is equivalently reformulated as a $\mathbf{u}-p$ system, 
which allows us to derive the fundamental solution and establish Green's identity. Under appropriate 
assumptions on the poroelastic parameters, we prove that the real and imaginary parts of the wave 
numbers are positive. The well-posedness of the exterior problem is also discussed.

We first introduce some basic notions to be used throughout this paper. Denote by $x=(x_1,x_2,x_3)^{\top}$ 
a generic point in $\R^3$. For a bounded domain $D$ with boundary by $\pa D$, its complementary set in $ \R^3$ 
is denoted by $D^c= \R^3\ba\ov{D}$. Let $B_r=\{x\in \R^3:|x|<r\}$ be the open ball of radius $r$ centered 
at origin with boundary $S_r=\{x\in \R^3:|x|=r\}$. In the time domain, the linearized poroelastic wave 
equations are \cite{Philippacopoulos1998Spectral}
\begin{align*}
\frac{\pa^2}{\pa t^2}(\rho\mathbf{u} +\rho_f \mathbf{w})&=\nabla \nabla \cdot [(\la_c+2\mu)\mathbf{u}
+\alpha M \mathbf{w}]-\mu \nabla \times \nabla \times \mathbf{u},\\
\frac{\pa^2}{\pa t^2}(\rho_f \mathbf{u}+m \mathbf{w})+ b \frac{\mathbf{w}}{\pa t}&=\nabla \nabla \cdot 
[\alpha M \mathbf{u}+ M \mathbf{w}],
\end{align*}
where the displacement field $\mathbf{u}$ represents the motion of the solid phase, and $\mathbf{w}$ 
is the pore fluid relative to the skeletal frame. By the time-harmonic convention $e^{-{\bf{i}}\omega t}$
for angular frequency and time, we take the complex valued functions $\mathbf{u}(x,t)=e^{-{\bf{i}}\omega t}
\mathbf{u}(x)$ and $\mathbf{w}(x,t)=e^{-{\bf{i}}\omega t}\mathbf{w}(x)$. This leads to the following 
time-harmonic Biot's system 
\begin{align}\label{poroelastic1}
\omega^2(\rho\mathbf{u} +\rho_f \mathbf{w}) &= -\nabla \nabla \cdot [(\la_c+2\mu)\mathbf{u}
+\alpha M \mathbf{w}]+\mu\nabla \times \nabla \times \mathbf{u},  \\ \label{poroelastic2}
\omega^2(\rho_f \mathbf{u}+m \mathbf{w})+{\bf{i}}\omega b \mathbf{w} &= -\nabla \nabla \cdot 
[\alpha M \mathbf{u}+ M \mathbf{w}].
\end{align}
The parameters for the Biot's system throughout this paper are listed in Table \ref{para_meaning}.
\begin{table}[htbp]
\centering
\caption{Physical meaning of the parameters for the poroelastic system}
\vspace{5pt}
\begin{tabular}{|c|c|} \hline
Notation & Physical meaning\\ \hline
$\la,\mu$ & Lam\'{e} constants \\ \hline
$\alpha$,$M$ & parameters of Biot's theory \\ \hline
$\la_c=\la+\alpha^2 M$ & modified Lam\'{e} constants \\ \hline
$\omega$ & frequency \\ \hline
$b$ &  dissipation coefficient \\ \hline
$\rho$ & mass density of bulk material \\ \hline
$\rho_f$ & mass density of pore fluid \\ \hline
$m$ & coefficient having units of mass denesity \\ \hline
\end{tabular}
\label{para_meaning}
\end{table}

Let $\Om$ be a bounded obstacle with smooth boundary $\pa \Om$. From the vector identity $\nabla \times 
\nabla \times \mathbf{u}=\nabla \nabla \cdot \mathbf{u}-\Delta\mathbf{u}$ and Biot's system 
\eqref{poroelastic1}-\eqref{poroelastic2}, the poroelastic scattering problem with homogeneous Dirichlet 
boundary condition can then be modeled by the following PDE boundary-value problem 
\begin{subequations}
\begin{align}\label{sporoelastic1}
\omega^2(\rho\mathbf{u} +\rho_f \mathbf{w})+\Delta_c^*\mathbf{u}+\alpha M \nabla \nabla \cdot 
\mathbf{w}={\bf Q} \quad &{\rm in} \quad \Om^c,\\\label{sporoelastic2}
\omega^2(\rho_f \mathbf{u}+m \mathbf{w})+{\bf{i}}\omega b \mathbf{w}+\nabla \nabla \cdot 
(\alpha M \mathbf{u}+ M \mathbf{w})=0\quad &{\rm in} \quad \Om^c,\\\label{sDBC2}
(\mathbf{u}^{\top},\bs{n}\cdot \mathbf{w})^{\top}=0 \quad&{\rm on}\quad \pa \Om,
\end{align}
\end{subequations}
where the source ${\bf Q}$ is assumed to have compact support inside $B_{r_0}$ with some $r_0>0$, and
the modified Lam\'{e} operator $\De_c^*$ is defined by
\begin{equation*}
\De_c^*\mathbf{u}:=\mu \De \mathbf{u}+(\la_c +\mu)\na \na \cdot \mathbf{u} =\na \cdot \sig_c(\mathbf{u}).
\end{equation*}
Here, the associated symmetric stress tensor $\sig_c(\mathbf{u})$ is given by the generalized Hooke's law
\begin{equation*}
\sig_c(\mathbf{u})=(\la_c \na \cdot \mathbf{u} )\mat{I}_3+2\mu \epsilon(\mathbf{u}), \quad 
\epsilon(\mathbf{u})=\frac{1}{2}(\na \mathbf{u}+(\na \mathbf{u})^{\top}).
\end{equation*}
Assume that a homogeneous poroelastic system admits a smooth solution $(\mathbf{u}^{\top},\mathbf{w}^{\top})^{\top}$ 
to the Biot's equations \eqref{sporoelastic1}-\eqref{sporoelastic2} in $B_{r_0}^c$ (with ${\bf Q}=0$). 
Using the Helmholtz decomposition
\begin{equation*}
\mathbf{u}=\nabla \phi_1 +\Psi_1,\;\mathbf{w}=\nabla \phi_2 +\Psi_2,\; \dive\;{\Psi_1}=\dive\;{\Psi_2}=0,
\end{equation*}
the Biot's equations  \eqref{sporoelastic1}-\eqref{sporoelastic2} are then transformed into the following equations
\begin{subequations}
\begin{align}
\Psi_2 &=\frac{\rho_f\omega({\bf{i}}b-\omega m)}{\omega^2m^2+b^2}\Psi_1,\label{HeD1} \\
\omega^2\rho\Psi_1+\frac{\omega\rho_f^2({\bf{i}}b-\omega m)}{\omega^2m^2+b^2}\Psi_1&=-\mu\Delta\Psi_1,	\label{HeD2}\\
\omega^2\rho\phi_1+\omega^2\rho_f\phi_2&=-[(\la_c+2\mu)\Delta\phi_1+\alpha M \Delta\phi_2],	\label{HeD3}\\
\omega^2\rho_f\phi_1+(\omega^2m+{\bf{i}}\omega b)\phi_2&=-(\alpha M \Delta\phi_1+ M \Delta\phi_2).	\label{HeD4}
\end{align}
\end{subequations}
Observing the equations \eqref{HeD3}-\eqref{HeD4}, we find that $\phi_1$ and $\phi_2$ satisfy the following equation
\begin{equation}\label{HeD4_1}
\Delta^2 \phi +\left(\frac{c_{11}+c_{22}}{c_{21}c_{12}-c_{11}c_{22}}\right) \Delta \phi
-\frac{1}{c_{21}c_{12}-c_{11}c_{22} }\phi=0,
\end{equation}
where
\begin{equation*}
\begin{aligned}
c_{11}& =-\frac{\begin{vmatrix}
\la_c+2\mu&	\omega^2\rho_f\\
\alpha M 	&\omega^2m+{\bf{i}}\omega b
\end{vmatrix}}{\begin{vmatrix}
\omega^2\rho&	\omega^2\rho_f\\
\omega^2\rho_f	&\omega^2m+{\bf{i}}\omega b
\end{vmatrix}},
& c_{12} &=-\frac{\begin{vmatrix}
\alpha M &	\omega^2\rho_f\\
M 	&\omega^2m+{\bf{i}}\omega b
\end{vmatrix}}{\begin{vmatrix}
\omega^2\rho&	\omega^2\rho_f\\
\omega^2\rho_f	&\omega^2m+{\bf{i}}\omega b
\end{vmatrix}},\\
c_{21}& =-\frac{\begin{vmatrix}
\omega^2\rho&	\la_c+2\mu\\
\omega^2\rho_f	&\alpha M 
\end{vmatrix}}{\begin{vmatrix}
\omega^2\rho&	\omega^2\rho_f\\
\omega^2\rho_f	&\omega^2m+{\bf{i}}\omega b
\end{vmatrix}},
& c_{22} &=-\frac{\begin{vmatrix}
\omega^2\rho&	\alpha M\\
\omega^2\rho_f	&	M 
\end{vmatrix}}{\begin{vmatrix}
\omega^2\rho&	\omega^2\rho_f\\
\omega^2\rho_f	&\omega^2m+{\bf{i}}\omega b
\end{vmatrix}}.
\end{aligned}
\end{equation*}
Rewrite the equation \eqref{HeD4_1} in the following form
\begin{equation}\label{HeD5}
(\Delta+\la_1^2) (\Delta+\la_2^2)\phi=0,
\end{equation}
where $\la_1$ and $\la_2$ satisfy the characteristic relation equations 
\begin{equation}\label{charac}
\la_1^2+\la_2^2=\frac{c_{11}+c_{22}}{c_{21}c_{12}-c_{11}c_{22}},\quad\la_1^2\la_2^2=
-\frac{1}{c_{21}c_{12}-c_{11}c_{22}}.
\end{equation}
Following the routine of \cite{Kupradze1979}, the poroelastic wave field $\mathbf{u}$ satisfies the decomposition lemma.
\begin{lemma}
The wave field $\mathbf{u}$ can be written in the form
$\mathbf{u}=\mathbf{u}^1+\mathbf{u}^2+\mathbf{u}^3,$
where $\mathbf{u}^i,i=1,2,3$ satisfy 
\begin{equation*}
\begin{aligned}
(\Delta +\la_1^2)\mathbf{u}^1 &=0, & (\Delta +\la_2^2)\mathbf{u}^2& =0, & (\Delta +\la_3^2)\mathbf{u}^3&=0, \\
{\rm curl}\mathbf{u}^1&=0, & {\rm curl}\mathbf{u}^2&=0, & \dive\mathbf{u}^3&=0.
\end{aligned}
\end{equation*}
Here, the wave number $\la_1,\la_2$ are the roots of the characteristic relation equations \eqref{charac}, 
and the complex wave number $\la_3$ is given by
\begin{equation*}
\la_3^2=\ds\frac{\omega^2\rho}{\mu}+\frac{\omega^3\rho_f^2({\bf{i}}b-\omega m)}{\mu(\omega^2m^2+b^2)}.
\end{equation*}
\end{lemma}
Throughout this paper, we assume that for $z\in\mathbb{C} $, the notation  $ z^{1/2} $ denotes the 
analytic branch of $\sqrt{z}$ satisfying $ \operatorname{Re}(z^{1/2})\geq 0$. This choice corresponds 
to taking the left half of the real axis as the branch cut in the complex plane.  For $z = z_1 + {\bf i} z_2$ 
with $z_1, z_2\in\R $, we have
\begin{equation}\label{z1/2}
z^{1/2} = \left(\frac{|z| + z_1}{2}\right)^{\frac{1}{2}} + {\bf i} \operatorname{sgn}(z_2)
\left(\frac{|z| - z_1}{2}\right)^{\frac{1}{2}}.
\end{equation}
For $z$ on the left half real axis, $ z^{1/2} $ is understood as the limit of  $(z + {\bf i}
 \varepsilon)^{1/2} $ as $ \varepsilon \to 0^+$.

The scattered wave field $\mathbf{u}$ and $\mathbf{w}$ are assumed to satisfy the Kupradze radiation 
conditions as $r=\sqrt{x_1^2+x_2^2+x_3^2} \rightarrow \infty$ for $i=1,2,3$ and $j=1,2$
\begin{equation}\label{Kupradze}
\begin{aligned}
&\mathbf{u}^j=o(r^{-1}), & &\pa_{x_i}\mathbf{u}^j=O(r^{-2}), \\
&\bs{n}\cdot\mathbf{w}=o(r^{-1}), & &\dive\mathbf{w}=O(r^{-2}), \\
&\mathbf{u}^3=o(r^{-1}), & &r(\pa_{r}\mathbf{u}^3-{\bf{i}}\la_3\mathbf{u}^3)=O(r^{-1}).
\end{aligned}
\end{equation}

For given $\bs{f}_s=(f_1,f_2,f_3)^{\top}\in H^{\frac{1}{2}}(S_{r_0})^3$ and $f_4\in H^{\frac{1}{2}}(S_{r_0})$,  
the following lemma establishes the uniqueness of solutions to the exterior Dirichlet problem
\begin{equation}\label{eDpp1}
\begin{cases}
\omega^2(\rho\mathbf{u} +\rho_f \mathbf{w})+\Delta_c^*\mathbf{u}+\alpha M \nabla \nabla \cdot \mathbf{w}=0 
\quad {\rm in}\;\;B_{r_0}^c,\\
\omega^2(\rho_f \mathbf{u}+m \mathbf{w})+{\bf{i}}\omega b \mathbf{w}+\nabla \nabla \cdot (\alpha M \mathbf{u}
+ M \mathbf{w})=0 \quad {\rm in}\;\;B_{r_0}^c,\\
(\mathbf{u}^{\top},	\bs{n}\cdot \mathbf{w})^{\top}=(\bs{f}_s^{\top},f_4)^{\top} \quad {\rm on} \;\;S_{r_0},\\
(\mathbf{u}^{\top},	\mathbf{w}^{\top})^{\top}\;\text{satisfies the Kupradze radiation conditions 
\eqref{Kupradze} at infinity}.
\end{cases}
\end{equation}
\begin{lemma}
The exterior Dirichlet problem \eqref{eDpp1} has at most one weak solution in $H_{loc}^1(B_{r_0}^c)^6$. 
\begin{proof}
It is equivalent to show that the homogeneous problem (i.e., with $\bs{f}_s = 0$ and $f_4 = 0$) admits 
only the trivial solution. Let $(\mathbf{u}^\top, \mathbf{w}^\top)^\top$ be any solution to this 
homogeneous problem, and denote its complex conjugate by $(\overline{\mathbf{u}}^\top, 
\overline{\mathbf{w}}^\top)^\top$. Direct calculation yields
\begin{equation}\label{conjugatee}
\begin{aligned}
0=\;&\omega^2(\rho\mathbf{u}\cdot \overline{\mathbf{u}} +\rho_f \mathbf{w}\cdot \overline{\mathbf{u}})
+\Delta_c^*\mathbf{u}\cdot \overline{\mathbf{u}}+\alpha M(\nabla \nabla \cdot  \mathbf{w})\cdot \overline{\mathbf{u}}\\
&+\omega^2(\rho_f \mathbf{u}\cdot \overline{\mathbf{w}}+m \mathbf{w}\cdot \overline{\mathbf{w}})
+{\bf{i}}\omega b \mathbf{w}\cdot \overline{\mathbf{w}}+\alpha M(\nabla \nabla \cdot  \mathbf{u})
\cdot \overline{\mathbf{w}}+ M(\nabla \nabla \cdot  \mathbf{w})\cdot \overline{\mathbf{w}}.  
\end{aligned} 
\end{equation}
Integrating both sides of \eqref{conjugatee} over $B_{R}\ba\ov{B_{r_0}}\;(R>r_0)$, and using the 
boundary condition on $S_{r_0}$, we obtain
\begin{align}\label{boundarye1}
&\int_{S_R}\left[\mathcal{T}_c\mathbf{u}\cdot\overline{\mathbf{u}} +\alpha M( \nabla \cdot 
\mathbf{w})(\bs{n}\cdot \overline{\mathbf{u}})+\alpha M(\nabla \cdot  \mathbf{u})(\bs{n} 
\cdot \overline{\mathbf{w}})+M( \nabla \cdot  \mathbf{w})(\bs{n}\cdot \overline{\mathbf{w}})\right]ds(x) \\\no
=&\int _{B_{R}\ba\ov{B_{r_0}}}\Big[-\omega^2(\rho|\mathbf{u}|^2 +\rho_f \mathbf{w}\cdot 
\overline{\mathbf{u}})+\mu\nabla\mathbf{u}:\nabla\overline{\mathbf{u}} +(\la_c+\mu)|\nabla 
\cdot \mathbf{u}|^2+\alpha M( \nabla \cdot  \mathbf{w})(\nabla\cdot \overline{\mathbf{u}})\\\no
&-\omega^2(\rho_f \mathbf{u}\cdot \overline{\mathbf{w}}+m |\mathbf{w}|^2)-{\bf{i}}\omega b 
|\mathbf{w}|^2+\alpha M(\nabla \cdot  \mathbf{u})(\nabla \cdot \overline{\mathbf{w}})+ 
M|\nabla \cdot  \mathbf{w}|^2\Big]dx,
\end{align}
where the traction operator $\mathcal{T}_c$ is given by $\mathcal{T}_c\mathbf{u}
=\mu\pa_{\bs{n}}\mathbf{u}+(\la_c+\mu)(\dive\mathbf{u})\bs{n}$ on $S_{R}$.
Taking the imaginary part on both sides of the equation \eqref{boundarye1} gives
\begin{align}\label{imagi_rela}
&-\omega b \int _{B_{R}\ba\ov{B_{r_0}}}|\mathbf{w}|^2dx\\\no
=&\Ip \int_{S_R}\left[\mathcal{T}_c\mathbf{u}\cdot\overline{\mathbf{u}} +\alpha M( \nabla \cdot
\mathbf{w})(\bs{n}\cdot \overline{\mathbf{u}})+\alpha M(\nabla \cdot  \mathbf{u})(\bs{n} \cdot
\overline{\mathbf{w}})+M( \nabla \cdot  \mathbf{w})(\bs{n}\cdot \overline{\mathbf{w}})\right]ds(x).  
\end{align}
The right-hand of \eqref{imagi_rela} tends to zero as $R$ tends to infinity from the Kupradze 
radiation condition \eqref{Kupradze}. Hence $\mathbf{w}$ must vanish identically. Setting 
$\mathbf{w}=0$ in \eqref{eDpp1} reduces the system to
\begin{equation}\label{eDpph2}
\begin{cases}
\Delta_c^*\mathbf{u}+\omega^2\rho\mathbf{u}=0\quad {\rm in}\;\;B_{r_0}^c,\\
\alpha M\nabla \nabla \cdot \mathbf{u}+\omega^2\rho_f \mathbf{u}=0\quad {\rm in}\;\;B_{r_0}^c,\\
\mathbf{u}=0\quad {\rm on} \;\;S_{r_0},\\
\mathbf{u}\;\text{satisfies the Kupradze radiation conditions \eqref{Kupradze} at infinity}.
\end{cases}
\end{equation}
From the classical elastic wave scattering theory \cite{BP2008}, it follows that the solution 
$\mathbf{u}$ to the system \eqref{eDpph2} must also vanish identically. The proof is thus complete.
\end{proof}	
\end{lemma}
To study the existence of solutions to the exterior Dirichlet problem \eqref{eDpp1}, we introduce 
an intermediate variable $p$ defined by
\begin{equation}\label{p_variable}
-p=\nabla\cdot[\alpha M \mathbf{u}+M \mathbf{w}].
\end{equation}
This, combined with the second equation in the system \eqref{eDpp1} gives
\begin{equation}\label{porop}
\mathbf{w}=\frac{1}{\omega^2m+{\bf{i}}\omega b}\nabla p -\frac{\omega\rho_f}{{\omega m+\bf{i}}b}\mathbf{u}.
\end{equation}
Taking the divergence on both sides of the equality \eqref{porop} yields
\begin{equation}\label{divporop}
\nabla\cdot \mathbf{w}=\frac{1}{\omega^2m+{\bf{i}}\omega b}\Delta p -\frac{\omega\rho_f}{{\omega m
+\bf{i}}b}\nabla \cdot \mathbf{u}.	
\end{equation}  
Using \eqref{p_variable}-\eqref{divporop} together with the first equation in \eqref{eDpp1} and 
eliminating $\mathbf{w}$, the problem \eqref{eDpp1} is transformed into the following $\mathbf{u}-p$ system:
\begin{equation}\label{laplaced1}
\begin{cases}
\ds\Delta^*\mathbf{u}+\Big(\omega^2 \rho -\frac{\omega^3 \rho_f^2}{{\omega m+\bf{i}}b}\Big)\mathbf{u}
+\Big(\frac{\omega\rho_f}{{\omega m+\bf{i}}b}-\alpha\Big)\nabla p=0 \quad {\rm in}\;\;B_{r_0}^c,\\
\ds\Delta p+(-\omega^2\rho_f+{\bf{i}}\alpha\omega b+\alpha\omega^2m)\dive\mathbf{u}+\frac{\omega^2m
+{\bf{i}}\omega b}{M}p=0 \quad {\rm in}\;\;B_{r_0}^c,\\
\ds\mathbf{u}=\bs{f}_s,\;\frac{\pa p }{\pa \bs{n}}=(\omega^2m+{\bf{i}}\omega b)f_4+
\omega^2\rho_f\bs{n}\cdot \bs{f}_s \quad {\rm on} \;\;S_{r_0},\\
(\mathbf{u}^{\top},p)^{\top}\;\text{satisfies the corresponding Kupradze radiation conditions at infinity}.
\end{cases}
\end{equation}
Here, the Kupradze radiation conditions for $\mathbf{u}-p$ system \eqref{laplaced1} can be similarly 
obtained by decomposing the wave fields as $\mathbf{u}=\mathbf{u}^1+\mathbf{u}^2+\mathbf{u}^3, p=p^1+p^2$; 
see \cite[Section 2]{zhang2022} for the detailed derivation. 

To study the problem \eqref{laplaced1}, we consider an abstract $\mathbf{u}-p$ system
\begin{subequations}
\begin{align}\label{abstractsystem1}
\Delta^*\mathbf{u}+T_1\mathbf{u}
+T_2\nabla p=0 \quad &{\rm in} \; B_{r_0}^c,\\\label{abstractsystem2}
\Delta p+T_3\dive\mathbf{u}+T_4p=0\quad &{\rm in} \;B_{r_0}^c,
\end{align}
\end{subequations}
where $T_i,i=1,2,3,4$ are general complex numbers.
In a block matrix form, the above system \eqref{abstractsystem1}-\eqref{abstractsystem2} is equivalently 
given by
\begin{equation}\label{matrix_form_poro}
B(\mathbf{u}^{\top},p)^{\top}=0\quad {\rm in}\; B_{r_0}^c,\quad{\rm where}\;B=\begin{bmatrix}
\Delta^*+T_1\mat{I}_3 & T_2\nabla \\
T_3\nabla^{\top} & \Delta +T_4\\
\end{bmatrix}.
\end{equation}
In what follows, we shall derive the fundamental solution of poroelastic system \eqref{matrix_form_poro} 
in the free space. A direct calculation shows that 
\begin{equation}
\begin{aligned}
\det(B)&=\mu^2(\la+2\mu)\Big(\Delta+\frac{T_1}{\mu}\Big)^2\Big[\Delta^2+\Big(T_4-\frac{T_2T_3}{\la+2\mu}
+\frac{T_1}{\la+2\mu}\Big)\Delta+\frac{T_1T_4}{\la+2\mu}\Big] \\
&\triangleq\mu^2(\la+2\mu)\Big(\Delta+\frac{T_1}{\mu}\Big)^2(\Delta+k_1^2)(\Delta+k_2^2),
\end{aligned}
\end{equation}
where $k_1$ and $k_2$ satisfy the characteristic equations
\begin{equation}\label{charac1}
k_1^2+k_2^2=T_4-\frac{T_2T_3}{\la+2\mu}+\frac{T_1}{\la+2\mu},\;k_1^2k_2^2=\frac{T_1T_4}{\la+2\mu}.
\end{equation}
The elements of the adjugate matrix $\operatorname{adj}B$ are listed in Appendix \ref{app:aBiot_fs}.
To seek a solution of
\begin{equation}\label{det_B}
\begin{aligned}
\det(B)\psi=\mu^2(\la+2\mu)\Big(\Delta+\frac{T_1}{\mu}\Big)^2(\Delta+k_1^2)(\Delta+k_2^2)\psi =0
\end{aligned}
\end{equation}
with an isolated singularity at the origin. Setting $\mu^2(\la+2\mu)(\Delta+\frac{T_1}{\mu})\psi=\Psi$ 
and $k_3^2=\frac{T_1}{\mu}$, then equation \eqref{det_B} reduces to
\begin{equation*}
\begin{aligned}
&\det(B)\psi=(\Delta+k_1^2)(\Delta+k_2^2)(\Delta+k_3^2)\Psi =0.
\end{aligned}
\end{equation*}
Consider the equation
\begin{equation}\label{fs_delta}
\begin{aligned}
(\Delta+k_1^2)(\Delta+k_2^2)(\Delta+k_3^2)\Psi = -\delta_0,
\end{aligned}
\end{equation}
in a suitable functional space, where $\delta_0$ is the Dirac function at origin.
Taking the Fourier transform on both sides of \eqref{fs_delta} yields
\begin{equation*}
(\lvert \xi \rvert^2-k_1^2)(\lvert \xi \rvert^2-k_2^2)(\lvert \xi \rvert^2-k_3^2)\hat{\Psi} =1.
\end{equation*}
For $\lvert \xi \rvert^2\neq k_1^2,\;k_2^2,\;k_3^2$, it therefore holds
\begin{equation}\label{Psi_ft}
\hat{\Psi} =\frac{A_1}{\lvert \xi \rvert^2-k_1^2}+\frac{A_2}{\lvert \xi \rvert^2-k_2^2}+
\frac{A_3}{\lvert \xi \rvert^2-k_3^2}.
\end{equation}
where the coefficients $A_1,A_2,A_3$ are given by 
\begin{equation}\label{A1A2A3}
A_1=\frac{1}{(k_1^2-k_2^2)(k_1^2-k_3^2)},\;A_2=\frac{1}{(k_2^2-k_1^2)(k_2^2-k_3^2)},
\;A_3=\frac{1}{(k_3^2-k_1^2)(k_3^2-k_2^2)}.
\end{equation}
In order to take the inverse Fourier transform on both sides of \eqref{Psi_ft}, we need to 
compute the integration
\begin{equation}\label{int_ift}
f(x)=\frac{1}{(2\pi)^3}\int_{ \R^3}\frac{e^{{\bf{i}}x\cdot \xi}}{\lvert \xi \rvert^2-k^2}d\xi,\quad \Ip k>0.
\end{equation}
By switching \eqref{int_ift} to spherical coordinates $(r=|\xi|,\theta,\phi)$, we obtain
\begin{equation}\label{integration_ift}
\begin{aligned}
f(x)&=\frac{1}{(2\pi)^2}\int_{0}^{+\infty}\int_{0}^{\pi}\frac{e^{{\bf{i}}\lvert x \rvert r \cos\theta}}
{r^2-k^2}r^2\sin\theta d\theta dr\\
&=\frac{1}{(2\pi)^2}\frac{1}{\lvert x \rvert}\int_{0}^{+\infty}\int_{0}^{\pi}\frac{{\bf{i}}r}{r^2-k^2}
\frac{\pa e^{{\bf{i}}\lvert x \rvert r\cos\theta}}{\pa \theta} d\theta dr\\
&=\frac{1}{(2\pi)^2}\frac{1}{\lvert x \rvert}\int_{0}^{+\infty}\frac{{\bf{i}}r(e^{-{\bf{i}}
\lvert x \rvert r}-e^{{\bf{i}}\lvert x \rvert r})}{r^2-k^2} dr\\
&=\frac{1}{(2\pi)^2}\frac{1}{{\bf{i}}|x|}\int_{-\infty}^{+\infty}\frac{r e^{{\bf{i}}|x|r}}{r^2-k^2}dr,
\end{aligned}
\end{equation}
where we extend the integration to the whole real line in the last step. 

The remaining integral in \eqref{integration_ift}
$$I = \int_{-\infty}^{+\infty}\frac{r e^{{\bf{i}}|x|r}}{r^2-k^2}dr$$
can be computed using the contour integration and residue theorem. The integrand has simple 
poles at $r=\pm k$. Because $\operatorname{Im}k>0$, the pole $r=k$ lies in the upper half-plane 
and $r=-k$ in the lower half-plane. Closing the contour in the upper half-plane 
(where $e^{{\bf{i}}|x|r}$ decays) gives
\begin{equation*}
I = 2\pi {\bf{i}}\;\underset{r=k}{\operatorname{Res}}\frac{r e^{{\bf{i}}|x|r}}{r^2-k^2}
=2\pi {\bf{i}}\lim_{r\to k}(r-k)\frac{r e^{{\bf{i}}|x|r}}{(r+k)(r-k)}={\bf{i}}\pi e^{{\bf{i}} k|x|}.
\end{equation*}
This, combined with \eqref{integration_ift} yields $f(x)=\ds \frac{e^{{\bf{i}}k\lvert x \rvert}}
{4\pi\lvert x \rvert}$. Hene
applying the inverse Fourier transform to \eqref{Psi_ft} yields 
\begin{equation*}
\Psi=A_1\frac{e^{{\bf{i}}k_1\lvert x \rvert}}{4\pi\lvert x \rvert}+A_2\frac{e^{{\bf{i}}k_2
\lvert x \rvert}}{4\pi\lvert x \rvert}+A_3\frac{e^{{\bf{i}}k_3\lvert x \rvert}}{4\pi\lvert x \rvert},
\end{equation*}
where $A_1,A_2,A_3$ are defined in \eqref{A1A2A3}.
Since 
\begin{equation*}
\det(B)\mat{I}_4\psi=B({\rm adj}B)\psi=B({\rm adj}B\psi)=0
\end{equation*}
holds everywhere except at the origin, each column of the matrix $\Phi(x)={\rm adj}B\psi$ satisfies 
the abstract homogeneous poroelastic system \eqref{matrix_form_poro}. Noting that all entries of 
${\rm adj} B$ contain the factor $\Delta+\frac{T_1}{\mu}$, and using the relation $\mu^2(\la+2\mu)
(\Delta+\frac{T_1}{\mu})\psi=\Psi$, the explicit form of the elements of $\Phi(x)$ can therefore be 
obtained (see Appendix \ref{app:aBiot_fs}). To this end, we define an associated operator of $B$ by 
\begin{equation*}
B^*=\begin{bmatrix}
\Delta^*+T_1\mat{I}_3 & -T_3\nabla \\
-T_2\nabla^{\top} & \Delta +T_4\\
\end{bmatrix}.
\end{equation*}
Following the same procedure to deduce the fundamental solution $\Phi$, one can obtain the fundamental 
solution $\Phi^*$ of the operator $B^*$. Now an important Green's identity will be proved in the following lemma.	
\begin{lemma}
For regular vectors $U=(\mathbf{u}^{\top},p)^{\top}$ and $V=(\mathbf{u}'^{\top},p')^{\top}$ defined in a 
bounded domain $D$ with the unit outward normal $\bs{n}$ to the smooth boundary $\pa D$, the following 
second Green’s identity holds
\begin{equation}\label{secondGreen}
\int_{D}(U^{\top}B^*V-V^{\top}BU)dx=\int_{\pa D}(U^{\top}R^*V-V^{\top}RU)ds(x),
\end{equation}
where 
\begin{equation}\label{R*R}
R^*=\begin{bmatrix}
\mathcal{T}\mathbb{I}_3 & -T_3\bs{n} \\
0 & \ds\frac{\pa}{\pa \bs{n}}\\
\end{bmatrix},\quad R=\begin{bmatrix}
\mathcal{T}\mathbb{I}_3 & T_2\bs{n} \\
0 & \ds\frac{\pa}{\pa \bs{n}}\\
\end{bmatrix},
\end{equation}
and the surface traction operator $\mathcal{T}$ is defined by $\mathcal{T}\mathbf{u}
=\mu\pa_{\bs{n}}\mathbf{u}+(\la+\mu)(\dive\mathbf{u})\bs{n}\;{\rm on}\;\pa D$.
\begin{proof}
A direct calculation gives
\begin{align}\no
&U^{\top}B^*V-V^{\top}BU\\ \no
=\;&\mathbf{u}\cdot(\Delta^*\mathbf{u}'-T_3\nabla p'+T_1\mathbf{u}')+p(\Delta p' +T_4 p'
-T_2 \dive \mathbf{u}') \\ \label{diffeq}
&-\mathbf{u}'\cdot(\Delta^*\mathbf{u}+T_1\mathbf{u}
+T_2\nabla p)-p'(\Delta p+T_3\dive\mathbf{u}+T_4p)\\ \no
=\;&(\mathbf{u}\cdot\Delta^*\mathbf{u}'-\mathbf{u}'\cdot\Delta^*\mathbf{u})+(p\Delta p'-p'\Delta p)
-T_2(\mathbf{u}'\cdot\nabla p+ p \dive \mathbf{u}')-T_3(\mathbf{u}\cdot\nabla p'+p'\dive\mathbf{u})\\ \no
=\;&(\mathbf{u}\cdot\Delta^*\mathbf{u}'-\mathbf{u}'\cdot\Delta^*\mathbf{u})+(p\Delta p'-p'\Delta p)-T_2 
\dive (p\mathbf{u}')-T_3\dive(p'\mathbf{u}).
\end{align}
Integrating the differential identity \eqref{diffeq} over $D$ and applying the divergence theorem 
yields the integral identity
\begin{align}\no
&\int_{D}(U^{\top}B^*V-V^{\top}BU)dx\\ \no
=&\int_{\pa D}(\mathbf{u}\cdot\mathcal{T}\mathbf{u}'-\mathbf{u}'\cdot\mathcal{T}\mathbf{u})ds
+\int_{\pa D}\Big(p\frac{ \pa p'}{\pa \bs{n}}-p'\frac{ \pa p}{\pa \bs{n}}\Big)ds-\int_{\pa D}
(T_2 p\bs{n})\cdot\mathbf{u}'ds-\int_{\pa D}(T_3p'\bs{n})\cdot\mathbf{u}ds\\ \label{inteq}
=&\int_{\pa D} \Big[\mathbf{u}\cdot\mathcal{T}\mathbf{u}'+p\frac{ \pa p'}{\pa \bs{n}}-(T_3p'\bs{n})
\cdot\mathbf{u}\Big]ds-\int_{\pa D}\Big[\mathbf{u}'\cdot\mathcal{T}\mathbf{u}+p'\frac{ \pa p}{\pa\bs{n} }
+(T_2 p\bs{n})\cdot\mathbf{u}'\Big]ds. 
\end{align}
With the definitions of $R^*$ and $R$ in \eqref{R*R}, equation \eqref{inteq} can be rearranged 
into the desired Green's identity \eqref{secondGreen}.	
\end{proof}
\end{lemma}
Assume the system \eqref{laplaced1} admits a unique weak solution $(\mathbf{u}^{\top},p)^{\top}$ in 
the product space $H_{loc}^1(\R^3\ba\overline{B_{r_0}})^3\times H_{loc}^2(\R^3\ba\overline{B_{r_0}})$. 
The solution of \eqref{laplaced1} yields the boundary traction  $\mathcal{T}\mathbf{u}\in { H^{-1/2}( S_{r_0})}$
and the trace $p\in { H^{3/2}( S_{r_0})}$.
It follows from the Green’s identity \eqref{secondGreen} with the parameters in 
\eqref{abstractsystem1}-\eqref{abstractsystem2} set to
\begin{equation}\label{T1234}
\begin{aligned} 
&T_1=\omega^2 \rho -\frac{\omega^3 \rho_f^2}{{\omega m+\bf{i}}b},\;
T_2=\frac{\omega\rho_f}{{\omega m+\bf{i}}b}-\alpha,\\
&T_3=-\omega^2\rho_f+{\bf{i}}\alpha\omega b+\alpha\omega^2m,\;
T_4=\frac{\omega^2m+{\bf{i}}\omega b}{M},
\end{aligned}
\end{equation}
that the solution $U=(\mathbf{u}^{\top},p)^{\top}$ of \eqref{laplaced1} has the integral representation
\begin{equation}\label{integrationrlaplaced}
\begin{aligned}
U(y)&=\int_{S_{r_0}}\Big[[R^*{\Phi}^*(x-y)]^{\top}U(x)-{\Phi}^*(x-y)^{\top}U'(x)\Big]ds(x),\;y\in B_{r_0}^c
\end{aligned}
\end{equation}
where $U|_{S_{r_0}}\in H^{1/2}(S_{r_0})^3\times H^{3/2}(S_{r_0})$ and
\begin{align*}
U'=\begin{bmatrix}
\mathcal{T}\mathbf{u}+\Big(\ds\frac{\omega\rho_f}{{\omega m+\bf{i}}b}-\alpha\Big)p\bs{n}  \\
\ds\frac{\pa p }{\pa \bs{n}} \\
\end{bmatrix}\in H^{-1/2}(S_{r_0})^3\times H^{1/2}(S_{r_0}).
\end{align*}
For given $\bs{f}=(\bs{f}_s^{\top},f_4)^{\top}\in H^{\frac{1}{2}}(S_{r_0})^4$ with 
$\bs{f}_s=(f_1,f_2,f_3)^{\top}\in H^{\frac{1}{2}}(S_{r_0})^3$, we denote by $E(\bs{f})
=[E_1(\bs{f})^{\top},E_2(\bs{f})^{\top}]^{\top}$ the solution of the $\mathbf{u}-\mathbf{w}$ 
system \eqref{eDpp1}, i.e.
\begin{equation}\label{E1E2}
E_1(\bs{f})=\mathbf{u},\;E_2(\bs{f}) = \frac{1}{{\bf i}\omega b+\omega^2m}\nabla p-\frac{\omega\rho_f}
{\omega m+{\bf i}b}\mathbf{u},
\end{equation}
where $U=(\mathbf{u}^{\top},p)^{\top}$ satisfies the integral representation \eqref{integrationrlaplaced}.

Denote $\bs{D}(x-y)=R^*{\Phi}^*(x-y)$ and $\bs{S}(x-y)={\Phi}^*(x-y)$. For $k=1,2,3$, let $M_{(k)}$ 
and $M^{(k)}$ be the $k$-th row vector and column vector, respectively, of matrix $M$ over $\mat{C}$. 
It holds from \eqref{integrationrlaplaced}-\eqref{E1E2} that
\begin{equation}\label{E1E2_element}
\begin{aligned}
E_1(\bs{f})(y)_{(k)}=&\int_{S_{r_0}}[\bs{D}(x-y)^{(k)}\cdot U(x)-\bs{S}(x-y)^{(k)}\cdot U'(x)]ds(x), \\
E_2(\bs{f})(y)_{(k)}=&\int_{S_{r_0}}[\bs{D_w}(x-y)^{(k)}\cdot U(x)-\bs{S_w}(x-y)^{(k)}\cdot U'(x)]ds(x),
\end{aligned}
\end{equation}
where
\begin{equation*}
\begin{aligned}
\bs{D_w}(x-y)^{(k)}&=\frac{1}{{\bf i}\omega b+\omega^2m}\frac{\pa\bs{D}(x-y)^{(4)}}{\pa y_{k}}
-\frac{\omega\rho_f}{\omega m+{\bf i}b}\bs{D}(x-y)^{(k)},\\
\bs{S_w}(x-y)^{(k)}
&=\frac{1}{{\bf i}\omega b+\omega^2m}\frac{\pa\bs{S}(x-y)^{(4)}}{\pa y_{k}}-\frac{\omega\rho_f}
{\omega m+{\bf i}b}\bs{S}(x-y)^{(k)}. 
\end{aligned}
\end{equation*}
All elements of the matrices $\bs{D}(x-y)$, $\bs{S}(x-y)$, $\bs{D_w}(x-y)$ and $\bs{S_w}(x-y)$ 
are of the form $\ds\frac{e^{{\bf i}k|x-y|}}{|x-y|^m}$ ($m\in\N^*$) multiplied by a multivariate 
polynomial in the variables
$$\begin{aligned}
\left\{ |x-y|,y_1,y_2,y_3,x_1,x_2,x_3,
\frac{x_1}{|x|},\frac{x_2}{|x|},\frac{x_3}{|x|} \right\}.
\end{aligned}$$
Now we can define the Dirichlet to Neumann (DtN) map 
$\mathcal{N}:H^{1/2}(S_{r_0})^4\to H^{-1/2}(S_{r_0})^4$ for problem \eqref{eDpp1} by 
\begin{equation}\label{porodtn}
\mathcal{N}\bs(\bs{f}_s^{\top},f_4)^{\top}=\begin{bmatrix}
\mathcal{T}_c & \alpha M {\bs n} \nabla^{\top} \\
\alpha M \nabla^{\top} &  M \nabla^{\top}
\end{bmatrix} \begin{bmatrix}
\mathbf{u}  \\
\mathbf{w} 
\end{bmatrix} \quad{\rm on}\;S_{r_0}.
\end{equation}	
The  well-posedness of \eqref{eDpp1} implies that $\mathcal{N}$ is a continuous linear operator.
Then the scattering problem \eqref{sporoelastic1}-\eqref{sDBC2} with the Kupradze radiation condition
\eqref{Kupradze} can be equivalently reduced to the boundary-value problem in a bounded domain 
$\Om_0 = B_{r_0} \ba\ov{\Om}$
\begin{subequations}
\begin{align}\label{sporoelastic1_1}
\omega^2(\rho\mathbf{u} +\rho_f \mathbf{w})+\Delta_c^*\mathbf{u}+\alpha M \nabla \nabla \cdot
\mathbf{w}={\bf Q} \quad &{\rm in} \quad \Om_0,\\\label{sporoelastic2_1}
\omega^2(\rho_f \mathbf{u}+m \mathbf{w})+{\bf{i}}\omega b \mathbf{w}+\nabla \nabla \cdot 
(\alpha M \mathbf{u}+ M \mathbf{w})=0\quad &{\rm in} \quad \Om_0,\\\label{sDBC1_1}
(\mathbf{u}^{\top},\bs{n}\cdot \mathbf{w})^{\top}=0 \quad&{\rm on}\quad \pa \Om,\\ \label{sDBC2_1}
\begin{bmatrix}
\mathcal{T}_c & \alpha M {\bs n} \nabla^{\top} \\
\alpha M \nabla^{\top} &  M \nabla^{\top}
\end{bmatrix}\begin{bmatrix}
\mathbf{u}  \\
\mathbf{w} 
\end{bmatrix}=\mathcal{N}(\mathbf{u}^{\top},\bs{n}\cdot \mathbf{w})^{\top} \quad&{\rm on}\quad S_{r_0}.
\end{align}
\end{subequations}
The following theorem on the positivity of the wave numbers is crucial for the subsequent analysis. 
It guarantees that the fundamental solution decays exponentially in the PML layer, which is the key 
ingredient for proving the exponential convergence of the PML method in subsection \ref{subsec:convergence}.
\begin{theorem}\label{wavenumber}
The wave numbers $k_1$, $k_2$ and $k_3$ of the system \eqref{laplaced1} satisfy the following characteristic relations:
\begin{equation}\label{k1k2}
\begin{aligned}
k_1^2+k_2^2
=\;&\frac{1}{\la+2\mu}\left(\omega^2\rho -\frac{\omega^4\rho_f^2 m}{\omega^2m^2+b^2}
+\frac{\omega^3\rho_f^2 b}{\omega^2m^2+b^2}{\bf{i}}\right)+\frac{\omega^2m+{\bf{i}}\omega b}{M}\\ 
&+\frac{1}{\la+2\mu}\frac{(\rho_f\omega^2m-\alpha \omega^2m^2-\alpha b^2-\rho_f\omega b {\bf{i}})
(\omega^2\rho_f -\alpha \omega^2m-\alpha \omega b{\bf{i}})}{\omega^2m^2+b^2}\\
k_1^2k_2^2
=\;&\frac{1}{\la+2\mu}\frac{\omega^4(\rho m-\rho_f^2)+\omega^3\rho b {\bf{i}}}{M},\\
k_3^2 =\;& \frac{1}{\mu}\left(\omega^2 \rho -\frac{\omega^3 \rho_f^2}{{\omega m+\bf{i}}b}\right) 
(\Ip k_3^2>0 \Longrightarrow \Ip k_3>0).
\end{aligned}
\end{equation}
Moreover, if $\omega$ is sufficiently small and the parameters satisfy
\begin{equation}\label{wavenumbercondition}
\frac{\rho}{\la+2\mu+\alpha^2M}\Big(\frac{\rho+\alpha^2m-2\alpha \rho_f}{\la+2\mu}-\frac{\rho}
{\la+2\mu+\alpha^2M}+\frac{m}{M}\Big)-\frac{\rho m-\rho_f^2}{M(\la+2\mu)}>0,
\end{equation}
then $\Rp k_j>0$ and $\Ip k_j>0$ for $j=1,2$.
\begin{proof}
Let $k_1^2+k_2^2=A_1(\omega)$ and $k_1^2k_2^2=A_2(\omega)$, and write
$A_1(\omega)=\omega^2f_1(\omega)+\omega g_1(\omega){\bf{i}}$, $A_2(\omega)
=\omega^4f_2(\omega)+\omega^3 g_2(\omega){\bf{i}}$, where
\begin{equation}\label{fg}
\begin{aligned}
f_1(\omega)=&\frac{1}{\la+2\mu}\Big(\rho -\frac{\omega^2\rho_f^2 m}{\omega^2m^2+b^2}\Big)
+\frac{1}{\la+2\mu}\frac{(\rho_f\omega^2m-\alpha \omega^2m^2-\alpha b^2)(\rho_f -\alpha m)}{\omega^2m^2+b^2}\\
&-\frac{1}{\la+2\mu}\frac{\rho_f\alpha b^2}{\omega^2m^2+b^2} + \frac{m}{M},\\
g_1(w)=&\frac{1}{\la+2\mu}\frac{\omega^2\rho_f^2 b}{\omega^2m^2+b^2}+\frac{ b}{M}
-\frac{1}{\la+2\mu}\frac{\rho_f b(\omega^2\rho_f -\alpha \omega^2m)}{\omega^2m^2+b^2}\\
&-\frac{1}{\la+2\mu}\frac{\alpha b(\rho_f\omega^2m-\alpha \omega^2m^2-\alpha b^2)}{\omega^2m^2+b^2},\\
f_2(\omega)=&\frac{1}{\la+2\mu}\frac{\rho m-\rho_f^2}{M},\;
g_2(w)=\frac{1}{\la+2\mu}\frac{\rho b }{M}.
\end{aligned}
\end{equation}
One readily verifies that
\begin{equation}\label{lim_f1g1}
\begin{aligned}
\ds\lim_{\omega \to 0}f_1(\omega)&=\frac{\rho+\alpha^2m-2\alpha \rho_f}{\la+2\mu} + \frac{m}{M},\\
\ds\lim_{\omega \to 0}g_1(w)&=\frac{ b}{M}
+\frac{\alpha^2 b}{\la+2\mu}>0.
\end{aligned}
\end{equation}
By virtue of Vieta's theorem, $k_1^2$ and $k_2^2$ are the two roots of the quadratic equation
\begin{equation}\label{Vieta}
t^2-A_1(\omega)t+A_2(\omega)=0.
\end{equation}
Using the quadratic formula, the roots of equation \eqref{Vieta} are given by:
\begin{equation*}
k_{1,2}^2=\frac{A_1(\omega)\pm[A_1^2(\omega)-4A_2(\omega)]^{\frac{1}{2}}}{2}.
\end{equation*}
Let $A_1^2(\omega)-4A_2(\omega)=\omega^2A(\omega)+\omega^3B(\omega){\bf{i}}$, where
\begin{equation}\label{AB}
A(\omega)=\omega^2f_1^2(\omega)-g_1^2(\omega)-4\omega^2f_2(\omega),\;
B(\omega)=2f_1(\omega)g_1(\omega)-4g_2(\omega).
\end{equation}
This, together with \eqref{z1/2} gives the principal square root of $A_1^2(\omega)-4A_2(\omega)$
\begin{equation}
[A_1^2(\omega)-4A_2(\omega)]^{\frac{1}{2}} = s(\om) + {\bf i}\frac{\om^3B(\om)}{2s(\om)},
\end{equation}
where $s(\om)$ is defined as
\begin{equation*}
s(\om)=\left(\frac{\omega^2A(\omega)+\left[\omega^4A^2(\omega)+\omega^6B^2(\omega)\right]
^{\frac{1}{2}}}{2}\right)^{\frac{1}{2}}.
\end{equation*}
Consequently, the imaginary part of $k_{1,2}^2$ is given by 
$$\Ip k_{1,2}^2 = \ds\frac{\om}{2}\left[g_1(\om)\pm\frac{\om^2B(\om)}{2s(\om)}\right].$$
Both imaginary parts of $k_1^2$ and $k_2^2$ are strictly positive if and only if 
\begin{equation}\label{equav_rela}
g_1(\omega) > \left| \ds\frac{\omega^2 B(\omega)}{2s(\omega)} \right| \Longleftrightarrow  
2g_1^2(\omega) + A(\omega) > [A^2(\omega) + \omega^2 B^2(\omega)]^{\frac{1}{2}}.
\end{equation}
It also follows obviously from \eqref{lim_f1g1} that
\begin{equation*}
\ds\lim_{\omega \to 0}A(\omega)=-\lim_{\omega \to 0}g_1^2(\om)=-\Big(\frac{ b}{M}
+\frac{\alpha^2 b}{\la+2\mu}\Big) ^2 < 0.
\end{equation*}
Noting that both sides of \eqref{equav_rela} are positive for small $\omega$, we substitute 
the expressions of $A$ and $B$ (see \eqref{AB}) to obtain
\begin{equation}\label{square}
\bigl(2g_1^2 + \omega^2 f_1^2 - g_1^2 - 4\omega^2 f_2\bigr)^2
> (\omega^2 f_1^2 - g_1^2 - 4\omega^2 f_2)^2 + \omega^2 (2f_1 g_1 - 4g_2)^2.
\end{equation}
A straightforward expansion and simplification of \eqref{square} leads to
\begin{equation}\label{square1}
g_1^2(f_1^2 - 4f_2) > (f_1 g_1 - 2g_2)^2 \Longleftrightarrow f_1 g_1 g_2 - g_1^2 f_2 - g_2^2 > 0.
\end{equation}
Taking the limit $\omega \to 0$ on \eqref{square1} and using \eqref{fg}-\eqref{lim_f1g1} yields 
the desired positivity criterion \eqref{wavenumbercondition}.  The positivity of the imaginary 
parts of $k_{1,2}^2$ guarantees  $\operatorname{Re} k_{1,2} > 0$ and $\operatorname{Im} k_{1,2} > 0$.
\end{proof}

\end{theorem}
\begin{remark}[A sufficient condition for \eqref{wavenumbercondition}]
{\rm We now give a sufficient condition for \eqref{wavenumbercondition} to hold. Observe that
\begin{equation}
\begin{aligned}
&\frac{\rho}{\la+2\mu+\alpha^2M}\left(\frac{\rho+\alpha^2m-2\alpha \rho_f}{\la+2\mu}-\frac{\rho}
{\la+2\mu+\alpha^2M}+\frac{m}{M}\right)-\frac{\rho m-\rho_f^2}{M(\la+2\mu)}\\ 
=&\frac{1}{M}\left[\frac{\rho^2+\alpha^2m\rho-2\alpha\rho \rho_f}{\alpha^2(\la+2\mu)
+\ds\frac{(\la+2\mu)^2}{M}}-\frac{\alpha^2(\rho m-\rho_f^2)}{\alpha^2(\la+2\mu)}+m
-\frac{\rho^2M}{(\la+2\mu+\alpha^2M)^2}\right].
\end{aligned}
\end{equation}
Since
\begin{equation*}
\begin{aligned}
& \ds\lim_{M\to \infty}\left[\frac{\rho^2+\alpha^2m\rho-2\alpha\rho \rho_f}{\alpha^2(\la+2\mu)
+\ds\frac{(\la+2\mu)^2}{M}}-\frac{\alpha^2(\rho m-\rho_f^2)}{\alpha^2(\la+2\mu)}+m-\frac{\rho^2M}
{(\la+2\mu+\alpha^2M)^2}\right]\\
=&\frac{\rho^2+\alpha^2m\rho-2\alpha\rho \rho_f}{\alpha^2(\la+2\mu)}-\frac{\alpha^2(\rho m-\rho_f^2)}
{\alpha^2(\la+2\mu)}+m=\frac{(\rho-\alpha\rho_f)^2}{\alpha^2(\la+2\mu)}+m>0,
\end{aligned}
\end{equation*}
we conclude that the condition \eqref{wavenumbercondition} is satisfied whenever $M$ is sufficiently large.}
\end{remark}
To conclude this section, we derive a variational formulation for the problem
\eqref{sporoelastic1_1}-\eqref{sDBC2_1}. We start by introducing a functional space on $\Om_0$:
\begin{equation*}
V_{\pa\Om}(\Om_0) = \left\{
(\bs{v}_1^\top, \bs{v}_2^\top)^\top \in H^1(\Om_0)^6 \;\middle|\;
\begin{aligned}
&\bs{v}_2 = \frac{1}{\omega^2 m+{\bf i}\omega b} \nabla q - \frac{\omega\rho_f}{\omega m+{\bf i}b} \bs{v}_1, 
\; q \in H^2(\Om_0), \\
&\bigl(\bs{v}_1^\top,\; \bs{n} \cdot \bs{v}_2\bigr)^\top = 0 \quad \text{on } \pa\Om,
\end{aligned}
\right\}.
\end{equation*}
Let $\Psi=\begin{bmatrix}
\Psi^s   \\
\Psi^f 
\end{bmatrix}\in V_{\pa\Om}(\Om_0)$ be an arbitrary test function. Multiplying \eqref{sporoelastic1} 
by $\overline{\Psi^s}$ and integrating
over $\Om_0$ gives
\begin{equation*}
\int_{\Om_0} \bigl[\,\omega^2(\rho\mathbf{u}+\rho_f\mathbf{w})
+ \Delta_c^*\mathbf{u} + \alpha M\nabla\nabla\cdot\mathbf{w}\,\bigr]\cdot\overline{\Psi^s}\,dx
= \int_{\Om_0} {\bf Q}\cdot\overline{\Psi^s}\,dx .
\end{equation*}
Integrating by parts, using the symmetry of $\sig_c(\mathbf{u})$
and the boundary conditions $\Psi^s = 0$ on $\pa\Om$, we obtain
\begin{equation*}
\begin{aligned}
\int_{\Om_0} \Delta_c^*\mathbf{u}\cdot\overline{\Psi^s}\,dx
&= -\int_{\Om_0} \bigl[\,\mu\nabla\mathbf{u}:\nabla\overline{\Psi^s}
+ (\la_c+\mu)(\nabla\cdot\mathbf{u})(\nabla\cdot\overline{\Psi^s})\,\bigr]\,dx
+ \int_{S_{r_0}} \mathcal{T}_c\mathbf{u}\cdot\overline{\Psi^s}\,ds, \\
\int_{\Om_0} \alpha M\nabla(\nabla\cdot\mathbf{w})\cdot\overline{\Psi^s}\,dx
&= -\int_{\Om_0} \alpha M(\nabla\cdot\mathbf{w})(\nabla\cdot\overline{\Psi^s})\,dx
+ \int_{S_{r_0}} \alpha M(\nabla\cdot\mathbf{w})\,(\bs{n}\cdot\overline{\Psi^s})\,ds.
\end{aligned}
\end{equation*}
Multiplying \eqref{sporoelastic2} by $\overline{\Psi^f}$ and integrating
over $\Om_0$ yields
\begin{equation*}
\int_{\Om_0} \bigl[\,\omega^2(\rho_f\mathbf{u}+m\mathbf{w})
+ {\bf i}\omega b\mathbf{w}
+ \nabla\nabla\cdot(\alpha M\mathbf{u}+M\mathbf{w})\,\bigr]\cdot\overline{\Psi^f}\,dx = 0 .
\end{equation*}
Integrating by parts, using $\bs{n}\cdot\Psi^f = 0$ on $\pa\Om$, we obtain
\begin{equation*}
\begin{aligned}
\int_{\Om_0} \nabla\nabla\cdot(\alpha M\mathbf{u}+M\mathbf{w})\cdot\overline{\Psi^f}\,dx
&= -\int_{\Om_0} \bigl[\,\alpha M(\nabla\cdot\mathbf{u})
+ M(\nabla\cdot\mathbf{w})\,\bigr](\nabla\cdot\overline{\Psi^f})\,dx \\
&\quad + \int_{S_{r_0}} \bigl[\,\alpha M(\nabla\cdot\mathbf{u})
+ M(\nabla\cdot\mathbf{w})\,\bigr](\bs{n}\cdot\overline{\Psi^f})\,ds .
\end{aligned}
\end{equation*}
Adding the resulting identities, we obtain the variational formulation: to find a solution
 $\Phi=(\mathbf{u}^{\top},\mathbf{w}^{\top})^{\top}\in V_{\pa\Om}(\Om_0)$ such that
\begin{equation}\label{scattering0}
\mathcal{B}(\Phi,\Psi)=-\int_{\Om_0}{\bf Q}\cdot \overline{\Psi^s}dx,\quad \forall\;\Psi=\begin{bmatrix}
\Psi^s   \\
\Psi^f 
\end{bmatrix} \in V_{\pa \Om}(\Om_0)
\end{equation}
where $\mathcal{B}(\Phi,\Psi)=\mathcal{B}_{\Om_0}(\Phi,\Psi)-\langle \mathcal{N}\Phi,\Psi_{\bs{n}}
\rangle_{S_{r_0}}$ with $\Psi_{\bs{n}}=\begin{bmatrix}
\Psi^s   \\
\bs{n}\cdot \Psi^f 
\end{bmatrix}$.
Here, the symmetric sesquilinear form $\mathcal{B}_{\Om_0}(\Phi,\Psi)$ is defined as
\begin{equation}\label{B_Om0}
\begin{aligned}
\mathcal{B}_{\Om_0}(\Phi,\Psi)=&\int_{\Om_0}\Big[\mu\nabla\mathbf{u}:\nabla\overline{\Psi^s} 
+(\la_c+\mu)(\nabla\cdot \mathbf{u})(\nabla\cdot\overline{\Psi^s})
\\
&+ M(\nabla\cdot \mathbf{w})(\nabla\cdot\overline{\Psi^f})+\alpha M(\nabla\cdot \mathbf{w})
(\nabla\cdot\overline{\Psi^s})+\alpha M(\nabla\cdot\mathbf{u})(\nabla\cdot\overline{\Psi^f})
\\
&-\omega^2(\rho\mathbf{u}\cdot\overline{\Psi^s} +\rho_f \mathbf{w}\cdot\overline{\Psi^s}
+\rho_f\mathbf{u}\cdot\overline{\Psi^f} +m\mathbf{w}\cdot\overline{\Psi^f})
-{\bf{i}}\omega b\mathbf{w}\cdot\overline{\Psi^f} \Big]dx,
\end{aligned}
\end{equation}
and $\langle\cdot,\cdot\rangle$ denotes the dual pair between the spaces $H^{-1/2}(S_{r_0})^4$ 
and $H^{1/2}(S_{r_0})^4$.

Assume in the following that the variational problem \eqref{scattering0} admits a unique solution. 
Then from \cite[Theorem 3.2]{Bramble2007}, we know that there exists a constant $C>0$ such that 
the inf-sup condition holds:
\begin{equation}\label{infsup_B1}
\sup_{\Psi\in { V_{\pa \Om}(\Om_0)}\ba\{0\}} \frac{\mathcal{B}(\Phi,\Psi)}
{\Vert \Psi\Vert_{H^{1}(\Om_0)^6}}\geq C\Vert \Phi\Vert_{H^{1}(\Om_0)^6},\quad \forall\;\Phi\in  V_{\pa \Om}(\Om_0).
\end{equation}
\begin{remark}[Constant Convention]
{\rm Here and in the sequel, the symbol $C$ denotes a generic positive constant which is allowed to 
change from one occurrence to the next.}
\end{remark}

\section{The spherical PML method}\label{sec:PML}
In this section, we shall derive the PML formulation for the poroelastic problem 
\eqref{sporoelastic1_1}-\eqref{sDBC2_1}. The well-posedness and stability of the PML problem 
in truncated domain and layer can be established under certain constraints on poroelastic parameters. 
Finally, we prove the exponential convergence of the PML method based on the PML extension 
and the exponential decay properties of the stretched fundamental solution.
\subsection{PML formulation}\label{subsec:PML form}
For given $r_2>r_0$, let $B_{r_2}$ be a spherical domain surrounding $B_{r_0}$.
The truncated PML domain and the PML layer are denoted by $\Om_2:=B_{r_2} \ba \overline{\Om}$ 
and $\Om_{\PML}:=B_{r_2} \ba \ov{B_{r_0}}$, respectively. 
For $x=(x_1,x_2,x_3)^{\top}\in\R^3$ consider the spherical coordinates
$x_1=r\sin\theta\cos\phi,\;x_2=r\sin\theta\sin\phi,\;
x_3=r\cos\theta$
with $r=|x|$ and Euler angle $(\theta,\phi)$.
Let $s(r)=1+{\bf i}\sig(r)$ be the model medium property which satisfies
\begin{equation*}
\sig\in C(\R^+),\;\sig\geq 0,\;{\rm and}\;\sig=0\;\;{\rm for}\;r\leq r_0.
\end{equation*}
Denote by $\wid{r}$ the complex radius defined as
\begin{equation}\label{stretched_r}
\wid{r}(r)=\int_0^{r}s(\tau)d\tau \triangleq r\beta(r) = r[1+{\bf i}\sig_m(r)],
\end{equation}
where $\sig_m(r)=\frac{1}{r}\int_0^r\sigma(\tau)d\tau$. For the Cartesian coordinates 
$x=(x_1,x_2,x_3)^{\top}$, the corresponding change of variables is $\wid{x}=(\wid{x}_1,\wid{x}_2,
\wid{x}_3)^{\top}$ with
\begin{equation}\label{stretched_x}
\wid{x}_1=\wid{r}\sin\theta\cos\phi,\;\wid{x}_2=\wid{r}\sin\theta\sin\phi,\;
\wid{x}_3=\wid{r}\cos\theta.
\end{equation}
Following \cite{BPT2010}, we impose the constraint conditions on $\sig_m\in C^2(\R^+)$
\begin{equation}\label{pmlcondition}
\begin{aligned}
&\sig_m(r)=0 \quad{\rm for}\; r \leq r_0,\\
&\sig_m(r)=\sig_0 \quad{\rm for}\; r \geq r_1,\\
&\sig_m(r)\; \text{is monotone  increasing for $r\in (r_0,r_1)$},
\end{aligned}
\end{equation}
where $\sig_0>0$ and $r_1$ are given constants satisfying
\begin{equation}\label{r1condition}
r_0 + \frac{3}{4}d < r_1 < r_2,\;d=r_2-r_0\;\text{denotes the thickness of PML layer}.
\end{equation}
The lower bound on $r_1$ guarantees that the cut‑off function used in the convergence analysis
 (see the proof of Lemma \ref{poroelasticPMLpotential}) is well defined.

From \eqref{stretched_r}-\eqref{pmlcondition}, one easily verifies that $\sig_m\leq\sig$ for 
all $r\geq r_0$. In the following text, we take $\sig_m(r)$ in the specific form
\begin{equation}\label{sigma_m}
\sig_m(r)=
\begin{cases}
0, & r \leq r_0,\\
\ds\sig_0\frac{h\Big(\frac{r-r_0}{r_1-r_0}\Big)}{h\Big(\frac{r-r_0}{r_1-r_0}\Big)
+kh\Big(\frac{r_1-r}{r_1-r_0}\Big)},\quad& r_0< r < r_1, \\
\sig_0,&r \geq r_1,
\end{cases}
\end{equation}
where $h(\tau)=\left\{
\begin{aligned}
&e^{-\frac{1}{\tau}},\;  &\tau>0\\
& 0,&\tau\leq 0
\end{aligned}
\right.$ and $0<k<1$.

Let $\{\bs e_r,\bs e_{\theta},\bs e_{\phi}\}$ be the unit vectors of the spherical coordinates
\begin{equation*}
\begin{aligned}
\bs{e}_{r} &=(\sin \theta \cos \phi, \sin \theta \sin \phi, \cos \theta)^\top, \\ 
\bs{e}_{\theta} &=(\cos \theta \cos \phi, \cos \theta \sin \phi,-\sin \theta)^\top, \\ 
\bs{e}_{\phi} &=(-\sin \phi, \cos \phi, 0)^\top.
\end{aligned}
\end{equation*}
For any scalar $f$ and vector $\mathbf{u}=(u_1,u_2,u_3)^{\top}$, the stretched differential 
operators in spherical coordinates are defined as follows
\begin{equation*}
\begin{aligned}
\wid{\nabla}f &=\frac{\pa f}{\pa \wid{r} }\bs{e}_r+\frac{1}{\wid{r}}\frac{\pa f}{\pa \theta }
\bs{e}_{\theta}+\frac{1}{\wid{r}\sin \theta}\frac{\pa f}{\pa \phi }\bs{e}_{\phi},\\
\wid{\dive}\mathbf{u} &=\wid{\nabla}\cdot\mathbf{u}=\frac{1}{\wid{r}^2}\frac{\pa(\wid{r}^2 u_r)}
{\pa \wid{r} }+\frac{1}{\wid{r}\sin\theta}\frac{\pa(\sin\theta u_{\theta})}{\pa \theta }
+\frac{1}{\wid{r}\sin \theta}\frac{\pa u_{\phi}}{\pa \phi },\\
\wid{\nabla}\times\mathbf{u} &=\frac{1}{ \wid{r}\sin\theta }\left[\frac{\pa(\sin\theta u_{\phi})}
{\pa \theta}-\frac{\pa u_{\theta}}{\pa \phi}\right]\bs{e}_r+\frac{1}{ \wid{r} }\left[\frac{1}{\sin\theta}
\frac{\pa u_{r}}{\pa \phi}-\frac{\pa( \wid{r}u_{\phi})}{\pa \wid{r}}\right]\bs{e}_{\theta}\\
&\quad+\frac{1}{ \wid{r} }\left[\frac{\pa( \wid{r}u_{\theta})}{\pa \wid{r}}-\frac{\pa u_{r}}
{\pa \theta}\right]\bs{e}_{\phi},
\end{aligned}
\end{equation*}
with $u_r=\mathbf{u}\cdot\bs e_r$, $u_{\theta}=\mathbf{u}\cdot\bs e_{\theta}$ and 
$u_{\phi}=\mathbf{u}\cdot\bs e_{\phi}$.
It is easy to verify that
\begin{equation}\label{relation_oper}
\wid{\nabla}f=P_1^{-1}\nabla f,\;\wid{\dive}\mathbf{u}=\frac{1}{\beta^2s}
\dive (P_2 \mathbf{u}),\;\wid{\nabla}\times\mathbf{u}=P_2^{-1}\nabla \times P_1\mathbf{u},
\end{equation}
where $P_1$ and $P_2$ are defined by
\begin{equation}\label{P1P2}
P_1=\Diag\{s,\beta,\beta\},\quad P_2=\Diag\{\beta^2,s\beta,s\beta\}.
\end{equation}
Moreover, it follows from a complex variable analyticity argument \cite{BPT2010} that
\begin{equation}\label{stretched_laplace}
\wid{\Delta}\mathbf{u}=\sum_{i=1}^3(\wid{\Delta}u_i)\bs{e}_i
=\sum_{i=1}^3\wid{\dive}(\wid{\nabla}u_i)\bs{e}_i=- \wid{\nabla}\times\wid{\nabla}\times u
+\wid{\nabla} \wid{\dive}\mathbf{u}.
\end{equation}
Here $u_i$ denotes the $i$'th component of $\mathbf{u}$ and $\bs{e}_i$ denotes the unit vector 
in the $i$'th direction. We also introduce the stretched modified Lam\'{e} operator 
$\wid{\De_c^*}$ defined by
\begin{equation*}
\wid{\De_c^*}\mathbf{u}=\mu \wid{\De} \mathbf{u}+(\la_c +\mu)\wid{\na} \wid{\na} \cdot \mathbf{u}.
\end{equation*}

For the complex stretched coordinates $\wid{x}=(\wid{x}_1,\wid{x}_2,\wid{x}_3)^{\top}$ 
and $\wid{y}=(\wid{y}_1,\wid{y}_2, \wid{y}_3)^{\top}$, we define the complex distance
\begin{equation}\label{complexdistance}
d(\wid{x},\wid{y})=\left[(\wid{x}_1-\wid{y}_1)^2+(\wid{x}_2-\wid{y}_2)^2
+(\wid{x}_3-\wid{y}_3)^2\right]^{\frac{1}{2}}.
\end{equation}
Given $\bs{f}=(\bs{f}_s^{\top},f_4)^{\top}\in H^{\frac{1}{2}}(S_{r_0})^4$, let the PML 
extension of $E(\bs{f})(y)$ (see \eqref{E1E2}) be defined as $\wid{E}(\bs{f})(y)=E(\bs{f})(\wid{y})$, 
where the term $|\wid{y}-x|$ is replaced by $\wid{d}(\wid{y},x)$.
This extension is obviously a solution to the exterior Dirichlet problem
\begin{equation}\label{PML_exterior}
\begin{cases}
\omega^2(\rho\wid{\mathbf{u}} +\rho_f \wid{\mathbf{w}})+\wid{\De_c^*}\wid{\mathbf{u}} 
+\alpha M\wid{\nabla} \wid{\nabla}\cdot  \wid{\mathbf{w}}=0\quad {\rm in}\;\;B_{r_0}^c,\\
\omega^2(\rho_f \wid{\mathbf{u}}+m \wid{\mathbf{w}})+{\bf{i}}\omega b \wid{\mathbf{w}}
+\wid{\nabla} \wid{\nabla} \cdot [\alpha M \wid{\mathbf{u}}+ M \wid{\mathbf{w}}]=0 
\quad {\rm in}\;\;B_{r_0}^c,\\
(\mathbf{u}^{\top},	\bs{n}\cdot \mathbf{w})^{\top}=(\bs{f}_s^{\top},f_4)^{\top}\quad {\rm on} \;\;S_{r_0}.
\end{cases}
\end{equation} 
From the definition of $\wid{E}(\bs{f})$ and \eqref{E1E2_element}, we know that it has the 
form $\wid{E}(\bs{f})=[\wid{E}_1(\bs{f})^{\top},\wid{E}_1(\bs{f})^{\top}]^{\top}$, where
\begin{equation}\label{E1E2_PML}
\begin{aligned}
\wid{E}_1(\bs{f})(y)_{(k)}=&\int_{S_{r_0}}[\bs{D}(x-\wid{y})^{(k)}\cdot U(x)
-\bs{S}(x-\wid{y})^{(k)}\cdot U'(x)]ds(x), \\
\wid{E}_2(\bs{f})(y)_{(k)}=&\int_{S_{r_0}}[\bs{D_w}(x-\wid{y})^{(k)}\cdot U(x)
-\bs{S_w}(x-\wid{y})^{(k)}\cdot U'(x)]ds(x).
\end{aligned}
\end{equation}
All elements of the matrices $\bs{D}(x-\wid{y})$, $\bs{S}(x-\wid{y})$, $\bs{D_w}
(x-\wid{y})$ and $\bs{S_w}(x-\wid{y})$  are of the form $\ds\frac{e^{{\bf i}
k\wid{d}(\wid{y},x)}}{\wid{d}(\wid{y},x)^m}$ ($m\in\N^*$) multiplied by a 
multivariate polynomial in the variables
$$\begin{aligned}
\left\{  \wid{d}(\wid{y},x),\wid{y}_1,\wid{y}_2,\wid{y}_3,x_1,x_2,x_3,
\frac{x_1}{|x|},\frac{x_2}{|x|},\frac{x_3}{|x|} \right\}.
\end{aligned}$$
By imposing a certain homogeneous boundary condition on $S_{r_2}$, we obtain the following truncated PML problem
\begin{equation}\label{poromodifiedPML}
\begin{cases}
\omega^2(\rho\wid{\mathbf{u}} +\rho_f \wid{\mathbf{w}})+\wid{\De_c^*}\wid{\mathbf{u}} +\alpha M\wid{\nabla} 
\wid{\nabla}\cdot  \wid{\mathbf{w}}=\mathbf{Q}\quad {\rm in}\;\;\Om_2,\\
\omega^2(\rho_f \wid{\mathbf{u}}+m \wid{\mathbf{w}})+{\bf{i}}\omega b \wid{\mathbf{w}}+\wid{\nabla} \wid{\nabla} 
\cdot [\alpha M \wid{\mathbf{u}}+ M \wid{\mathbf{w}}]=0 \quad {\rm in}\;\;\Om_2,\\
(\mathbf{u}^{\top},	\bs{n}\cdot \mathbf{w})^{\top}=0\quad {\rm on} \;\;\pa \Om,\\
(\mathbf{u}^{\top},	\bs{n}\cdot \mathbf{w})^{\top}=0\quad {\rm on} \;\;S_{r_2}.
\end{cases}
\end{equation}

\subsection{The PML equation in truncated domain and layer}\label{subsec:PML equation}
\setcounter{equation}{0}
In this subsection, we shall show that the PML system \eqref{poromodifiedPML} in the truncated domain
has a unique weak solution in the space $H_0^1( \Om_2)^6$ under appropriate constraints 
on the poroelastic and PML parameters. We begin with a key lemma concerning the PML Helmholtz decomposition.
\begin{lemma}[PML Helmholtz decomposition]\label{PMLHD}
Let $D=B_{r_2}\ba\ov{\Om}$ or $D=B_{r_2}\ba\ov{B_{r_0}}$ be the truncated PML domain or the PML layer.
Define $V(D)=\{\bs{v}\in H^1(D)^3:\bs{n} \cdot\bs{v} =0 \;{\rm on}\;\pa D\}
$. The following two decompositions hold:
\begin{enumerate}
\item[(1)] For any $\bs{U}\in H_0^1(D)^3$, there exists $\theta_1 \in H^2(D)\cap H_0^1(D)$ and $\Phi_1\in H^1(D)^3$ 
such that $\bs{U}= \wid{\nabla}\theta_1 +\Phi_1$ and $\wid{\dive}\Phi_1=0$.
\item[(2)] For any $\bs{W}\in V(D)$, there exists $\theta_2 \in H^2(D)/\mathbb{C}$ and $\Phi_2\in V(D)$ such that 
$\bs{W}= \wid{\nabla}\theta_2 +\Phi_2$, $\wid{\dive}\Phi_2=0$ and $\bs{n}\cdot\wid{\nabla}\theta_2=0$ on $\pa D$.
\end{enumerate}
\end{lemma}
\begin{proof}
(1) For any fixed $\bs{U}\in H_0^1(D)^3$, consider the Dirichlet problem
\begin{equation}\label{eq:dirichlet}
\wid{\Delta}\theta = \wid{\dive}\bs{U} \quad\text{in } D, \qquad \theta = 0 \quad\text{on } \pa D .
\end{equation}
From the relations of the stretched operators \eqref{relation_oper}-\eqref{stretched_laplace}, 
we have the identity
\begin{equation}
\wid{\Delta}\theta = \frac{1}{s\beta^2}\dive\bigl(P_2P_1^{-1}\nabla\theta\bigr).
\end{equation}
This leads to the weak formulation of \eqref{eq:dirichlet}: to find $\theta\in H_0^1(D)$ such that
\begin{equation*}
a(\theta,\psi):=\int_D P_2P_1^{-1}\nabla\theta\cdot\nabla\overline{\psi}\,dx
= -\int_D s\beta^2\,(\wid{\dive}\bs{U})\,\overline{\psi}\,dx
\quad \forall \psi\;\in H_0^1(D).
\end{equation*}
It follows from the definitions of $P_1$ and $P_2$ (see \eqref{P1P2}) that
\begin{equation}\label{eq:variation}
\begin{aligned}
\Rp a(\psi,\psi)&=\Rp\int_{D}P_2P_1^{-1}\nabla \psi\cdot\nabla\overline{\psi}dx \\
&= \Rp\int_{D}\Diag\{\beta^2s^{-1},s,s\}\nabla \psi\cdot\nabla\overline{\psi}dx \\
&=\Rp\int_{D}\left[\beta^2s^{-1}\vert (\nabla \psi)_r \vert^2+s\vert (\nabla \psi)_\theta\vert^2
+s\vert (\nabla \psi)_\phi \vert^2\right]dx.
\end{aligned}
\end{equation}
From the definition of $\beta$ in \eqref{stretched_r} and the fact that $\sig_m\le\sig$ for $r\ge r_0$, we obtain
\begin{equation}\label{bs_lower_bound}
\Rp(\beta^2s^{-1}) = \Rp\frac{(1+{\bf i}\sig_m)^2}{1+{\bf i}\sig} = \frac{1-\sig_m^2+2\sig\sig_m}{1+\sig^2} 
\geq \frac{1}{1+\Vert\sig\Vert_{L^{\infty}(B_{r_2})}^2}.
\end{equation}
This, combined with \eqref{eq:variation} yields
\begin{equation}\label{lowerbound_a}
\Rp a(\theta,\theta) \geq
\frac{\lVert \nabla \theta \rVert_{{L^2(D)^3}}^2}{1+\Vert\sig\Vert_{L^{\infty}(B_{r_2})}^2}.
\end{equation}
By the Poincar\'e inequality and elliptic regularity theory, problem \eqref{eq:dirichlet} 
admits a unique solution $\theta_1\in H^2(D)\cap H_0^1(D)$. Setting $\Phi_1 = \bs{U} - 
\wid{\nabla}\theta_1$ gives the desired decomposition.

(2) For any fixed $\bs{W}\in V(D)$, consider the Neumann problem
\begin{equation}\label{eq:neumann}
\wid{\Delta}\theta=\wid{\dive}\bs{W} \quad\text{in } D, \qquad \bs{n}\cdot \wid{\nabla}\theta=0 
\quad\text{on } \pa D.
\end{equation}
The compatibility conditions are satisfied since
\begin{equation*}
\begin{aligned}
\int_{D}\dive(P_2P_1^{-1}\nabla\theta)dx &=\int_{\pa D}\bs{n}\cdot P_2P_1^{-1}\nabla\theta ds(x)= 
\int_{\pa D}\frac{\beta^2}{s}\bs{n}\cdot \nabla\theta ds(x)=0\\
\int_{D}\dive(P_2\bs{W})dx&=\int_{\pa D}\bs{n}\cdot P_2\bs{W}ds(x)= \int_{\pa D}\beta^2\bs{n}\cdot\bs{W}ds(x)=0.
\end{aligned}
\end{equation*}
Proceeding as in the first part, the Lax-Milgram theorem and elliptic regularity theory implies the 
 existence of the unique solution
$\theta_2 \in H^2(D)/\C$ to the problem \eqref{eq:neumann}. Defining $\Phi_2 = \bs{W} - \wid{\nabla}\theta_2$ 
completes the proof.
\end{proof}
For $D=B_{r_2}\ba\ov{\Om}$ or $D=B_{r_2}\ba\ov{B_{r_0}}$, we introduce the sesquilinear form similar to
 $\mathcal{B}_{\Om_0}$ in \eqref{B_Om0}
\begin{align}\label{poropmlbiform_B}
\mathcal{B}_{D}(\Phi,\Psi)=&\int_{D}\Bigg\{\mu P_2P_1^{-1}\nabla\Phi^s:\nabla\overline{\Psi^s} 
+\frac{1}{\beta^2s}\Big[(\la_c+\mu)(\nabla\cdot P_2\Phi^s)(\nabla\cdot P_2\overline{\Psi^s})\\ \no
&+\alpha M(\nabla\cdot P_2\Phi^f)(\nabla\cdot P_2\overline{\Psi^s})+\alpha M(\nabla\cdot P_2\Phi^s)
(\nabla\cdot P_2\overline{\Psi^f})+ M(\nabla\cdot P_2\Phi^f)(\nabla\cdot P_2\overline{\Psi^f})\Big]\\ \no
&-\omega^2\beta^2s(\rho\Phi^s\cdot\overline{\Psi^s} +\rho_f \Phi^f\cdot\overline{\Psi^s}+\rho_f\Phi^s
\cdot\overline{\Psi^f} +m\Phi^f\cdot\overline{\Psi^f})-{\bf i}\omega b\beta^2s\Phi^f\cdot\overline{\Psi^f} \Bigg\}dx
\end{align}
where $\Phi=\begin{bmatrix}
\Phi^s \\
\Phi^f
\end{bmatrix}$, $\Psi=\begin{bmatrix}
\Psi^s \\
\Psi^f
\end{bmatrix}$. 
Now we are in the position to present the main result of this subsection.
\begin{theorem}\label{maintheorem}
For given $\bs{P}=(\bs{U}^{\top},\bs{V}^{\top})^{\top}\in H_0^1(D)^3\times V(D)$, the variational problem
\begin{equation}\label{mainweakequation}
\begin{aligned}
\mathcal{B}_{D}(\Phi,\Psi) =(\beta^2s\bs{P},\Psi )_D
,\;\forall\;\Psi \in H_0^1(D)^3\times V(D)
\end{aligned}
\end{equation}
has a unique weak solution $\Phi\in H_0^1(D)^3\times V(D)$ under the smallness assumption on $\omega$
 and $\alpha$, where $(\cdot,\cdot)_D$ denotes the inner product of $L^2(D)$.
\begin{proof} 
It suffices to show the system 
\begin{equation}\label{pmlporoelastic1}
\begin{aligned}
\wid{\De_c^*}\Phi^s+\alpha M \wid{\nabla}\wid{\nabla}\cdot\Phi^f+\omega^2\rho \Phi^s 
+\omega^2\rho_f \Phi^f= \bs{U}& \quad\text{in }  D,\\
\wid{\nabla}\wid{\nabla}\cdot(\alpha M\Phi^s+ M \Phi^f)
+ \omega^2(\rho_f \Phi^s+m \Phi^f)+{\bf i}\omega b \Phi^f= \bs{V}& \quad\text{in }  D,\\
(\Phi^s,\bs{n}\cdot\Phi^f)=0 & \quad\text{on } \pa D,
\end{aligned}
\end{equation}
has a unique weak solution $(\Phi^s,\Phi^f)$ in $H_0^1(D)^3\times V(D)$.

We first prove the existence. By Lemma \ref{PMLHD} we decompose
\begin{equation}\label{PMLHDsources}
\bs{U} = \wid{\nabla}\theta_1 + \Phi_1,\quad
\bs{V} = \wid{\nabla}\theta_2 + \Phi_2,
\end{equation}
with $\theta_1\in H^2(D)\cap H_0^1(D)$ and $\theta_2\in H^2(D)/\mathbb{C}$.

i) \textit{Scalar potential system.}
Introduce the Sobolev space
\begin{equation*}
V_1(D) = \left\{\bs{v}=(v_1,v_2)\in H^1(D)^2 : \ds\int_D v_1dx = \int_D v_2dx = 0\right\}.
\end{equation*}
Motivated by the derivation of equations \eqref{HeD3}-\eqref{HeD4} via the Helmholtz decomposition, 
we consider the system
\begin{align}
(\la_c+2\mu)\wid{\Delta}\phi_1 + \alpha M\wid{\Delta}\phi_2 + \omega^2\rho\phi_1 + \omega^2\rho_f\phi_2 
= \theta_1  &\quad\text{in }  D,\label{biharmonic1}\\
\alpha M\wid{\Delta}\phi_1 + M\wid{\Delta}\phi_2 + \omega^2(\rho_f\phi_1+m\phi_2) + {\bf i}\omega b\phi_2 
= \theta_2  &\quad\text{in }  D, \label{biharmonic2}\\
\bs{n}\cdot\wid{\nabla}\phi_1 =\bs{n}\cdot\wid{\nabla}\phi_2 = 0 &\quad\text{on } \pa D. 
\label{biharmonicboundaryconditions}
\end{align}
For $\bs{p}=(\phi_1,\phi_2)$, $\bs{q}=(\psi_1,\psi_2)\in V_1(D)$, we introduce the sesquilinear 
form for the system \eqref{biharmonic1}-\eqref{biharmonicboundaryconditions} 
\begin{equation*}
\begin{aligned}
\mathcal{B}_{1}(\bs{p},\bs{q})=&\int_{D}\Big[ (\la_c+2\mu) P_2P_1^{-1}\nabla\phi_1\cdot
\nabla\overline{\psi_1} +\alpha M P_2P_1^{-1}\nabla\phi_2\cdot\nabla\overline{\psi_1}+
\alpha M P_2P_1^{-1}\nabla\phi_1\cdot\nabla\overline{\psi_2}\\
& + M P_2P_1^{-1}\nabla\phi_2\cdot\nabla\overline{\psi_2} -\omega^2\beta^2s
(\rho\phi_1\cdot\overline{\psi_1} +\rho_f \phi_2\cdot\overline{\psi_1}+
\rho_f\phi_1\cdot\overline{\psi_2} +m\phi_2\cdot\overline{\psi_2})\\
&-{\bf i}\omega b\beta^2s\phi_2\cdot\overline{\psi_2}\Big] dx.
\end{aligned}
\end{equation*}
Taking the real part of the sesquilinear form $\mathcal{B}_{1}$ yields
\begin{align}
\Rp\mathcal{B}_{1}(\bs{p},\bs{p})=&\Rp\int_{D} \Big[(\la_c+2\mu) P_2P_1^{-1}|\nabla\phi_1|^2
+ M P_2P_1^{-1}|\nabla\phi_2|^2\Big]dx \no\\ \label{realpart_B1}
&+\Rp\int_{D}\alpha MP_2P_1^{-1}\left(\nabla\phi_2\cdot\nabla\overline{\phi_1}+\nabla\phi_1
\cdot\nabla\overline{\phi_2}\right)dx\\ \no
& -\Rp\int_{D}\left[\omega^2\beta^2s(\rho\lvert\phi_1\rvert^2 +\rho_f \phi_2\cdot
\overline{\phi_1}+\rho_f\phi_1\cdot\overline{\phi_2} +m\lvert\phi_2\rvert^2)-{\bf i}
\omega b\beta^2s|\phi_2|^2\right] dx.
\end{align}
In what follows, we shall estimate \eqref{realpart_B1} term by term. 
Similar to \eqref{lowerbound_a}, we obtain the following inequalities
\begin{equation}
\begin{aligned}
&\Rp\int_{D} \Big[(\la_c+2\mu) P_2P_1^{-1}|\nabla\phi_1|^2+ M P_2P_1^{-1}|\nabla\phi_2|^2\Big]dx\\
&\geq \frac{(\la_c+2\mu)\lVert \nabla \phi_1 \rVert_{{L^2(D)^3}}^2+M\lVert 
\nabla \phi_2 \rVert_{{L^2(D)^3}}^2}{1+\Vert\sig\Vert_{L^{\infty}(B_{r_2})}^2},
\end{aligned}
\end{equation}
and
\begin{equation}\label{estimate_realpartB1}
\begin{aligned}
&\Rp\int_{D}\alpha MP_2P_1^{-1}\left(\nabla\phi_2\cdot\nabla\overline{\phi_1}+
\nabla\phi_1\cdot\nabla\overline{\phi_2}\right)dx\\
\leq &\int_{D}\alpha M\max\{\Rp(\beta^2s^{-1}),1\}(|\nabla \phi_1|^2+|\nabla \phi_2|^2)dx\\
\leq &\; 2\alpha M\left(\lVert \nabla \phi_1 \rVert_{{L^2(D)^3}}^2+\lVert \nabla \phi_2 
\rVert_{{L^2(D)^3}}^2\right),
\end{aligned}
\end{equation}
where we have used the estimate in the last inequality
\begin{equation*}
\Rp(\beta^2s^{-1}) = \Rp\frac{(1+{\bf i}\sig_m)^2}{1+{\bf i}\sig} 
= \frac{1-\sig_m^2+2\sig\sig_m}{1+\sig^2}
\leq \frac{1+2\sig^2}{1+\sig^2}\leq 2.
\end{equation*}
For the last term of \eqref{realpart_B1}, it holds that
\begin{equation}\label{estimate_realpartB1_3}
\begin{aligned}
&\Rp\int_{D}\left[\omega^2\beta^2s(\rho\lvert\phi_1\rvert^2 +\rho_f \phi_2
\cdot\overline{\phi_1}+\rho_f\phi_1\cdot\overline{\phi_2} +m\lvert\phi_2\rvert^2)
-{\bf i}\omega b\beta^2s|\phi_2|^2\right] dx\\
\leq & \int_{D} \left\{\om^2\Rp(\beta^2s)\left[(\rho+\rho_f)\lvert\phi_1\rvert^2
+(m+\rho_f)\lvert\phi_2\rvert^2\right]+\omega b\Ip(\beta^2s)|\phi_2|^2 \right\}dx\\
\leq & \int_{D} \left\{\om^2\left[(\rho+\rho_f)\lvert\phi_1\rvert^2+(m+\rho_f)
\lvert\phi_2\rvert^2\right]+3\omega b\|\sig\|_{L^{\infty}(B_{r_2})}|\phi_2|^2 \right\}dx\\
= &\; \om^2(\rho+\rho_f)\|\phi_1\|_{L^2(D)}^2+\om\left[\om(m+\rho_f)+
3 b\|\sig\|_{L^{\infty}(B_{r_2})}\right]\|\phi_2\|_{L^2(D)}^2,
\end{aligned}
\end{equation}
where we have used the estimates in the second inequality
\begin{equation*}
\begin{aligned}
\Rp(\beta^2s) &= \Rp[(1+{\bf i}\sig_m)^2(1+{\bf i}\sig)] 
= 1-\sig_m^2-2\sig\sig_m\leq 1,\\
\Ip(\beta^2s) &= \Ip[(1+{\bf i}\sig_m)^2(1+{\bf i}\sig)] 
= (1-\sig_m^2)\sig+2\sig_m\leq 3\|\sig\|_{L^{\infty}(B_{r_2})}.
\end{aligned}
\end{equation*}  
Combining \eqref{realpart_B1}-\eqref{estimate_realpartB1_3}, it follows that 
$$\begin{aligned}
\Rp\mathcal{B}_{1}(\bs{p},\bs{p})
\geq &\frac{(\la_c+2\mu)\lVert \nabla \phi_1 \rVert_{{L^2(D)^3}}^2
+M\lVert \nabla \phi_2 \rVert_{{L^2(D)^3}}^2}{1+\Vert\sig\Vert_{L^{\infty}
(B_{r_2})}^2}-2\alpha M\left(\lVert \nabla \phi_1 \rVert_{{L^2(D)^3}}^2+
\lVert \nabla \phi_2 \rVert_{{L^2(D)^3}}^2\right)\\
&-\om^2(\rho+\rho_f)\|\phi_1\|_{L^2(D)}^2-\om\left[\om(m+\rho_f)+
3 b\|\sig\|_{L^{\infty}(B_{r_2})}\right]\|\phi_2\|_{L^2(D)}^2
\end{aligned}
$$ 
Therefore, for small $\alpha$ and $\omega$, it follows from the Poincaré inequality that
$$\begin{aligned}
&\Rp\mathcal{B}_{1}(\bs{p},\bs{p})\geq C( \lVert \phi_1\rVert_{H^1(D)}^2+
\lVert \phi_2\rVert_{H^1(D)}^2)=C\lVert \bs{p} \rVert_{V_1(D)}^2.
\end{aligned}
$$ 
Thus, the system \eqref{biharmonic1}-\eqref{biharmonicboundaryconditions} 
admits a unique solution $(\phi_1,\phi_2) \in V_1(D)$. Moreover, the standard 
elliptic regularity theory gives $(\phi_1,\phi_2) \in H^4(D)$.

ii) \textit{Vector potential system.} Consider the system for ${\Psi_1}$ and ${\Psi_2}$:
\begin{align}\label{maxwell1}
\mu\wid{\Delta} \Psi_1+\Big(\omega^2\rho -\frac{\omega^3\rho_f^2}{\omega m+{\bf i}b}\Big)\Psi_1
=\Phi_1-\frac{\omega\rho_f}{\omega m+{\bf i}b}\Phi_2 & \quad\text{in }  D,\\ \label{maxwell2}
\Psi_2=\frac{\Phi_2}{\omega^2m+{\bf i}\omega b}-\frac{\omega\rho_f}{\omega m+
{\bf i}b}\Psi_1 & \quad\text{in }  D,\\ \label{pmlmaxwellboundaryconditions}
\Psi_1=-\wid{\nabla}_{\bs{\tau}}\phi_1 &\quad \text{on } \pa D.
\end{align}
Here $\phi_1$ is the solution to the system \eqref{biharmonic1}-\eqref{biharmonicboundaryconditions} 
and $\wid{\nabla}_{\bs{\tau}}\phi_1$ is the tangential component of the vector
$\wid{\nabla}\phi_1$ on $\pa D$.

For $\Psi_1,\Psi'\in H^1(D)^3,$ we introduce the corresponding sesquilinear form
\begin{equation*}
\mathcal{B}_{2}(\Psi_1,\Psi')=\int_{D}\left[\mu P_2P_1^{-1}\nabla\Psi_1:\nabla\overline{\Psi^{'}} 
- \Big(\omega^2\rho -\frac{\omega^3\rho_f^2}{\omega m+{\bf i}b}\Big)\beta^2s\Psi_1\cdot \overline{\Psi^{'}}\right]dx,
\end{equation*} 
which can be easily verified to be coercive in $H_0^1(D)^3$ for small $\om$. Thus problem
\eqref{maxwell1}-\eqref{pmlmaxwellboundaryconditions} has a unique solution $\Psi_1\in H^2(D)^3$ 
by the Lax-Milgram theorem and elliptic regularity theory.  The following integral identity holds 
from \cite[equation (3.5)]{BPT2010}
\begin{equation}\label{int_iden}
\int_{D}\beta^2s(\wid{\Delta}\Psi_1)\cdot(\wid{\nabla}\phi)dx=\int_{D}\beta^2s(\wid{\nabla}
(\wid{\nabla}\cdot\Psi_1 ))\cdot(\wid{\nabla}\phi)dx, \quad \forall \phi\in C_0^{\infty}(D).
\end{equation}
Multiplying both sides of \eqref{maxwell1} by $\beta^2s\wid{\nabla}\overline{\phi}$ and integrating 
by parts, using \eqref{int_iden} and conditions $\wid{\nabla}\cdot \Phi_1=\wid{\nabla}\cdot \Phi_2=0$, 
we obtain that $\wid{\nabla}\cdot \Psi_1$ satisfies
\begin{equation}
\int_{D}P_2P_1^{-1}\nabla(\wid{\nabla}\cdot\Psi_1 )\cdot(\nabla\overline{\phi})dx+\Big(\omega^2\rho 
-\frac{\omega^3\rho_f^2}{\omega m+{\bf i}b}\Big)\int_{D}\beta^2s(\wid{\nabla}\cdot\Psi_1)\overline{\phi}dx=0.
\end{equation}
For sufficiently small $\omega$, it follows from the definition of $P_2P_1^{-1}$ and 
\eqref{bs_lower_bound} that $\wid{\nabla}\cdot \Psi_1=0$ in $D$, then using \eqref{maxwell2} 
yields $\wid{\nabla}\cdot \Psi_2=0$ in $D$.
Therefore, using the identity $\wid{\Delta}\Psi_1=-\wid{\nabla}\times\wid{\nabla}\times\Psi_1$ 
(since $\wid{\nabla}\cdot\Psi_1=0$), and direct calculations yield an equavilent form of the 
system \eqref{maxwell1}-\eqref{pmlmaxwellboundaryconditions}
\begin{align}\label{eq:Phi1}
-\mu\wid{\nabla}\times\wid{\nabla} \times \Psi_1+\omega^2(\rho\Psi_1 +\rho_f \Psi_2)
=\Phi_1 & \quad\text{in }  D,\\\label{eq:Phi2}
\omega^2(\rho_f \Psi_1+m \Psi_2)+{\bf i}\omega b \Psi_2=\Phi_2 & \quad\text{in }  D,\\ \label{eq:Phi_bc}
\Psi_1=-\wid{\nabla}_{\bs{\tau}}\phi_1 &\quad {\rm on}\;\pa D.
\end{align}
iii) \textit{Construction of the full solution.} 
Taking $\wid{\nabla}$ on both sides of \eqref{biharmonic1}-\eqref{biharmonic2} and using the divergence
free condition on $\Psi_{1,2}$ gives
\begin{align}\label{nablatheta1}
(\la_c+2\mu) \wid{\nabla}\wid{\nabla}\cdot(\wid{\nabla} \phi_1+\Psi_1)+\alpha M \wid{\nabla}
\wid{\nabla}\cdot(\wid{\nabla} \phi_2+\Psi_2)+\omega^2\rho \wid{\nabla}\phi_1 +\omega^2\rho_f 
\wid{\nabla}\phi_2&= \wid{\nabla}\theta_1,\\ \label{nablatheta2}
\alpha M \wid{\nabla}\wid{\nabla}\cdot(\wid{\nabla} \phi_1+\Psi_1)+ M \wid{\nabla}\wid{\nabla}
\cdot(\wid{\nabla} \phi_2+\Psi_2)+ \omega^2(\rho_f \wid{\nabla}\phi_1+m \wid{\nabla}\phi_2)
+{\bf i}\omega b \wid{\nabla}\phi_2&= \wid{\nabla}\theta_2.
\end{align}
Adding \eqref{eq:Phi1} to \eqref{nablatheta1}, and \eqref{eq:Phi2} to \eqref{nablatheta2} and 
using \eqref{PMLHDsources}, it follows that 
$\Phi^s = \wid{\nabla}\phi_1 + \Psi_1$ and $\Phi^f = \wid{\nabla}\phi_2 + \Psi_2$ solve the 
PML system \eqref{pmlporoelastic1}.

Now we prove the uniqueness. Assume $\mathcal{B}_{D}(\Phi,\Psi)=0$ for all $\Psi\in H_0^1(D)^3\times V(D)$.
For the given data $\bs{P}=\ov{\Phi}/\beta^{2}s \in H_0^1(D)^3\times V(D)$, there exists a solution
$\wid{\Phi}$ to the variational problem \eqref{mainweakequation}.
Taking $\Psi=\ov{\Phi}$ in \eqref{mainweakequation} yields
\begin{equation*}
\left(\beta^2s\frac{\ov{\Phi}}{\beta^2s},\ov{\Phi} \right)_D= (\overline{\Phi},\overline{\Phi} )_D
= \mathcal{B}_{D}(\wid{\Phi},\overline{\Phi}) = \mathcal{B}_{D}(\Phi,\overline{\wid{\Phi}})=0.
\end{equation*}	
Hence $\Phi = 0$, which completes the proof.
\end{proof}

\end{theorem}
In what follows, we consider a PDE system in the PML layer for the convergence analysis of the PML method:
\begin{subequations}
\begin{align}\label{PMLlayerstabilityequa1}
\omega^2(\rho\wid{\mathbf{u}} +\rho_f \wid{\mathbf{w}})+\wid{\De_c^*}\wid{\mathbf{u}} +\alpha M\wid{\nabla} 
\wid{\nabla}\cdot  \wid{\mathbf{w}}=0
\quad &{\rm in} \;\Om_{\PML} , \\ \label{PMLlayerstabilityequa22}
\omega^2(\rho_f \wid{\mathbf{u}}+m \wid{\mathbf{w}})+{\bf{i}}\omega b \wid{\mathbf{w}}+\wid{\nabla} \wid{\nabla} 
\cdot [\alpha M \wid{\mathbf{u}}+ M \wid{\mathbf{w}}]=0 
\quad &{\rm in} \;\Om_{\PML},\\\label{PMLlayerconditionouter}
(\wid{\mathbf{u}}^{\top},	\bs{n}\cdot\wid{\mathbf{w}})^{\top} =(\bs{F}_s^{\top},F_4)^{\top} \quad &{\rm on} 
\; S_{r_2},\\ \label{PMLlayerconditioninner}
(\wid{\mathbf{u}}^{\top},	\bs{n}\cdot\wid{\mathbf{w}})^{\top}=0 \quad &{\rm on} \; S_{r_0},
\end{align}
\end{subequations}
where  $\bs{F}_s\in H^{\frac{1}{2}}(S_{r_2})^3$ and $F_4\in H^{\frac{1}{2}}(S_{r_2})$.
Then the variational formulation of \eqref{PMLlayerstabilityequa1}-\eqref{PMLlayerconditioninner} 
is as follows: given $(\bs{F}_s^{\top},F_4)^{\top}\in H^{1/2}(S_{r_2})^4$, find $	(\wid{\mathbf{u}}^{\top},
\wid{\mathbf{w}}^{\top})^{\top} \in 
H_{S_{r_0}}^1(\Om_{\PML})^6=:\{\Phi=(\Phi_s^{\top},\Phi_f^{\top})^{\top}\in H^1(\Om_{\PML})^6: (\Phi_s^{\top},
\bs{n}\cdot \Phi_f)^{\top}=0\;{\rm on}\;S_{r_0}\}$ such that
\begin{equation}\label{varform_BPML}
\mathcal{B}_{\Om_\PML}((\wid{\mathbf{u}}^{\top},	 \wid{\mathbf{w}}^{\top})^{\top} ,\Psi)=0, \quad\forall\;
\Psi\in H_0^1(\Om_{\PML})^6.
\end{equation}

\begin{theorem}\label{pmllayer}
Let $D=\Om_{\PML}$ and assume the parameters satisfy the constraints in the theorem \ref{maintheorem}. 
For given $(\bs{F}_s^{\top},F_4)^{\top}\in H^{1/2}(S_{r_2})^4$, the system \eqref{PMLlayerstabilityequa1}
-\eqref{PMLlayerconditioninner}
admits a unique weak solution $(\wid{\mathbf{u}}^{\top}, \wid{\mathbf{w}}^{\top})^{\top}\in H^{1}_{S_{r_0}}
(\Om_{\PML})^6$ such that $(\wid{\mathbf{u}}^{\top},\bs{n}\cdot\wid{\mathbf{w}})^{\top}
=(\bs{F}_s^{\top},F_4)^{\top}$ on $S_{r_2}$. Moreover, the following stability estimate holds for the system 
\eqref {PMLlayerstabilityequa1}-\eqref {PMLlayerconditioninner}
\begin{equation}\label{stability_pmllayer}
\Vert (\wid{\mathbf{u}}^{\top},	 \wid{\mathbf{w}}^{\top})^{\top} \Vert_{H^1(\Om_{\PML})^6}
\leq C(d)\Vert (\bs{F}_s^{\top},F_4)^{\top} \Vert_{H^{1/2}(S_{r_2})^4},
\end{equation}
where the constant $C(d)$ depends on the thickness $d$ and grows at most polynomially in $d$.
\begin{proof}
The existence and uniqueness follow by decomposing the problem into scalar and vector potential 
systems exactly as in Theorem \ref{maintheorem}. Specifically, the solution $(\Phi^s,\Phi^f)$ of 
\eqref{PMLlayerstabilityequa1}-\eqref{PMLlayerconditioninner} can be constructed by
\begin{equation}
\Phi^s= \wid{\nabla}\phi_1 +\Psi_1,\;\Phi^f= \wid{\nabla}\phi_2 +\Psi_2\quad{\rm in}\;\Om_{\PML},
\end{equation}
where the scalar potential $(\phi_1,\phi_2)\in V_1(\Om_{\PML})$ satisfies the PDE system 
\eqref{biharmonic1}-\eqref{biharmonicboundaryconditions}
with $\theta_1=\theta_2=0$ and $\bs{n}\cdot \wid{\nabla}\phi_1=\bs{n}\cdot\wid{\nabla}\phi_2=0$ on $S_{r_0}$,
$\bs{n}\cdot \wid{\nabla}\phi_1=\bs{n}\cdot\bs{F}_s,\;\bs{n}\cdot\wid{\nabla}\phi_2=F_4$ on $S_{r_2}$, 
and the vector potential $(\Psi_1,\Psi_2)$ satisfies \eqref{eq:Phi1}-\eqref{eq:Phi_bc} with $\Phi_1=\Phi_2=0$ 
and $\Psi_1=-\wid{\nabla}_{\bs{\tau}}\phi_1$ on $S_{r_0}$,
$\Psi_1=\bs{F}_{s\tau}-\wid{\nabla}_{\bs{\tau}}\phi_1$ on $S_{r_2}$. 

Moreover, by the standard regularity and stability theory for strongly elliptic systems on smooth domains 
(see, e.g., \cite[Theorems 4.16 and 4.18]{McLean2000}), we obtain the $H^2$ estimate 
\begin{align}\no
\Vert \phi_1\Vert_{H^2(\Om_\PML)}+\Vert\phi_2\Vert_{H^2(\Om_\PML)}&\leq C(\Vert \phi_1\Vert_{H^1(\Om_\PML)}
+\Vert\phi_2\Vert_{H^1(\Om_\PML)}\\\label{layerbiharmonicellipiticestimates}
&\quad+\Vert \bs{n}\cdot\bs{F}_s\Vert_{H^{1/2}(S_{r_2})}+\Vert F_4\Vert_{H^{1/2}(S_{r_2} )})\\ \no
&\leq C(\Vert \bs{n}\cdot\bs{F}_s\Vert_{H^{1/2}(S_{r_2})}+\Vert F_4\Vert_{H^{1/2}(S_{r_2} )}),
\end{align}
and the stability estimate for $\Psi_1$
\begin{equation}\label{layermaxwellellipticstab}
\|\Psi_1\|_{H^1(\Om_\PML)} \leq C\left(\|\bs{F}_{\tau}-\wid{\nabla}_{\bs{\tau}}\phi_1\|_{H^{1/2}(S_{r_2})^3}
+\|\wid{\nabla}_{\bs{\tau}}\phi_1\|_{H^{1/2}(S_{r_0})^3}\right).
\end{equation}

We now establish the stability estimate for the problem \eqref{PMLlayerstabilityequa1}-\eqref{PMLlayerconditioninner}. 
First, using triangle inequality yields
\begin{equation*}
\begin{aligned}
\Vert\Phi^s \Vert_{H^1(\Om_\PML)^3}&\leq	\Vert\wid{\nabla}\phi_1  \Vert_{H^1(\Om_\PML)^3}
+\Vert\Psi_1 \Vert_{H^1(\Om_\PML)^3}	
\leq C(d)(\Vert \phi_1\Vert_{H^2(\Om_\PML)}+\Vert\Psi_1 \Vert_{H^1(\Om_\PML)^3}),\\
\Vert\Phi^f \Vert_{H^1(\Om_\PML)^3}&\leq	\Vert\wid{\nabla}\phi_2  \Vert_{H^1(\Om_\PML)^3}
+\Vert\Psi_2 \Vert_{H^1(\Om_\PML)^3}	
\leq C(d)(\Vert \phi_2\Vert_{H^2(\Om_\PML)}+\Vert\Psi_2 \Vert_{H^1(\Om_\PML)^3}).
\end{aligned}
\end{equation*}
Next, we use estimates \eqref{layerbiharmonicellipiticestimates} and \eqref{layermaxwellellipticstab} to obtain
\begin{align}\label{layerstabability1}
\Vert\Phi^s \Vert_{H^1(\Om_\PML)^3}+\Vert\Phi^f \Vert_{H^1(\Om_\PML)^3}\leq\;&C\big(\Vert 
\bs{n}\cdot\bs{F}_s\Vert_{H^{1/2}(S_{r_2})}+\Vert F_4\Vert_{H^{1/2}(S_{r_2} )}\\ \no
&+\Vert \bs{F}_{s\tau}-\wid{\nabla}_{\bs{\tau}}\phi_1\Vert_{H^{1/2}(S_{r_2} )^3}\big).
\end{align}
Using the triangle inequality and the trace theorem yields
\begin{equation}\label{layermaxwelltrace}
\begin{aligned}
\Vert \bs{F}_{s\tau}-\wid{\nabla}_{\bs{\tau}}\phi_1\Vert_{H^{1/2}(S_{r_2} )^3}&\leq
\Vert \bs{F}_{s\tau}\Vert_{H^{1/2}(S_{r_2})^3}+\Vert \wid{\nabla}_{\bs{\tau}}\phi_1\Vert_{H^{1/2}(S_{r_2} )^3}\\
&\leq \Vert \bs{F}_{s\tau}\Vert_{H^{1/2}(S_{r_2})^3}+\Vert \wid{\nabla}\phi_1\Vert_{H^{1}(\PML)^3} \\
&\leq \Vert \bs{F}_{s\tau}\Vert_{H^{1/2}(S_{r_2})^3}+C(d)\Vert \phi_1\Vert_{H^{2}(\PML)} .
\end{aligned}
\end{equation}
Finally, substituting inequality \eqref{layermaxwelltrace} into \eqref{layerstabability1} and observing that
\begin{equation*}
\Vert \bs{n}\cdot\bs{F}_s\Vert_{H^{1/2}(S_{r_2})}+ \Vert \bs{F}_{s\tau}\Vert_{H^{1/2}(S_{r_2})^3}
\leq C \Vert \bs{F}_{s}\Vert_{H^{1/2}(S_{r_2})^3},
\end{equation*}
we obtain the desired estimate \eqref{stability_pmllayer}.

\end{proof}
\end{theorem}

\subsection{Convergence analysis of the PML method}\label{subsec:convergence}
In this subsection, we shall show the convergence analysis of the PML method. We start with 
the following elementary lemma \cite[lemma 2]{CCZ2013}.
\begin{lemma}\label{keyestimate}
For any $z_j=a_j+{\bf i}b_j$ with $a_j,b_j\in\R$, $j=1,2,3$, such that $a_1b_1+a_2b_2+a_3b_3\geq 0$ and $a_1^2+a_2^2+a_3^2>0$, we have
$$\Ip(z_1^2+z_2^2+z_3^2)^{\frac{1}{2}}\geq \frac{a_1b_1+a_2b_2+a_3b_3}{\sqrt{a_1^2+a_2^2+a_3^2}}.$$
\end{lemma}
For $x\in S_{r_0},y\in S_{r}$ with $r_0<r<r_2$, it follows from the stretched spherical coordinates 
\eqref{stretched_x} that $\wid{y} = (\wid{r}\sin\theta\cos\phi, \wid{r}\sin\theta\sin\phi, 
\wid{r}\cos\theta)^{\top}$ and $x = (r_0\sin\theta_0\cos\phi_0, r_0\sin\theta_0\sin\phi_0, r_0\cos\theta_0)^{\top}$.
Recalling the definition of the complex distance \eqref{complexdistance}, we have 
$\wid{d}(\wid{y},x)=(z_1^2+z_2^2+z_3^2)^{\frac{1}{2}}$, where
\begin{equation*}
\begin{aligned}
z_1&=r\sin\theta\cos\phi-r_0\sin\theta_0\cos\phi_0+{\bf i}r\sin\theta\cos\phi\sig_m(r)\triangleq a_1+{\bf i}b_1,\\
z_2&=r\sin\theta\sin\phi-r_0\sin\theta_0\sin\phi_0+{\bf i}r\sin\theta\sin\phi\sig_m(r)\triangleq a_2+{\bf i}b_2,\\
z_3&=r\cos\theta-r_0\cos\theta_0+{\bf i}r\cos\theta\sig_m(r)\triangleq a_3+{\bf i}b_3.
\end{aligned}
\end{equation*}
For the sake of clarity, we introduce the vectors $\vec{A}=r\bs{e}_r,\vec{B}=r_0\bs{e}_r$, 
and $\vec{C}=r_0\bs{e}_r^0$, where the unit vectors 
$\bs{e}_r=(\sin\theta\cos\phi,\sin\theta\sin\phi,\cos\theta)^{\top}$ and $\bs{e}_r^0
=(\sin\theta_0\cos\phi_0,\sin\theta_0\sin\phi_0,\cos\theta_0)^{\top}$; see Figure 
\ref{Cont-sig18} for the geometric configuration.
\begin{figure}[ht]
\begin{center}
\includegraphics[width=2.5in,height=2.5in]{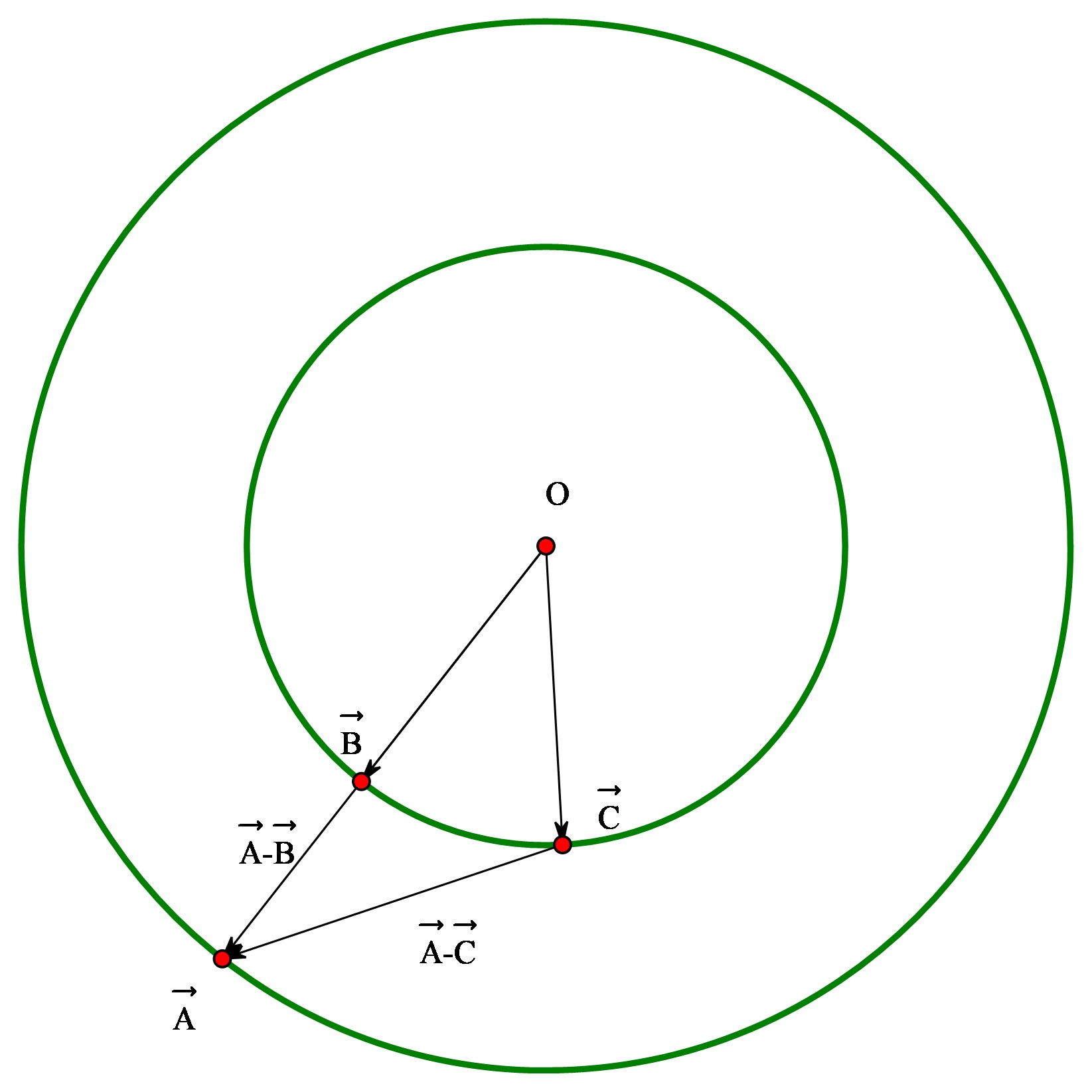}
\caption{Geometrical configuration of vectors $\vec{A},\vec{B},\vec{C}$.
}\label{Cont-sig18}
\end{center}
\end{figure}
Then we have         
\begin{equation*}
(\vec{A}-\vec{C})\cdot(\vec{A}-\vec{B}) = (r-r_0)(r - r_0\bs{e}_r\cdot\bs{e}_r^0),
\end{equation*}
hence
\begin{equation*}
\sum_{j=1}^3 a_j b_j = r\sig_m(r)\,\frac{(\vec{A}-\vec{C})\cdot(\vec{A}-\vec{B})}{r-r_0}
= r\sig_m(r)(r - r_0\bs{e}_r\cdot\bs{e}_r^0).
\end{equation*}
Since $\bs{e}_r\cdot\bs{e}_r^0\le 1$, we have $r - r_0\bs{e}_r\cdot\bs{e}_r^0\ge r-r_0$ 
and therefore
\begin{equation*}
\sum_{j=1}^3 a_j b_j \ge r\sig_m(r)(r-r_0).
\end{equation*}
This, combined with Lemma 6.1 yields
\begin{equation}\label{imd_lower}
\Ip\wid{d}(\wid{y},x)
\ge \frac{\sum_{j=1}^3 a_j b_j}{\sqrt{a_1^2+a_2^2+a_3^2}}
= \frac{\sum_{j=1}^3 a_j b_j}{|\vec{A}-\vec{C}|}\geq\frac{r\sig_m(r)(r-r_0)}{r_0+r}.
\end{equation}
Moreover, an estimate for the modulus of $\wid{d}(\wid{y},x)$ follows the triangle inequality
\begin{equation*}
|\wid{d}(\wid{y},x)|^2\le \sum_{j=1}^3(a_j^2+b_j^2)
\le |\vec{A}-\vec{C}|^2 + [r\sig_m(r)]^2 \le (r+r_0)^2 + r^2\sig_0^2.
\end{equation*}
Under the assumptions for parameters of the $\mathbf{u}-p$ system in Theorem \ref{wavenumber}, 
we have $\Rp k_j,\Ip k_j>0$ for $j=1,2,3$. This, combined the estimate \eqref{imd_lower} yields
$$ \begin{aligned}
\lvert e^{{\bf i}k_j\wid{d}(\wid{y},x)}\rvert
\leq \lvert e^{-\Rp k_j\Ip\wid{d}(\wid{y},x)}\rvert
\leq e^{-k_m\frac{r\sig_m(r)(r-r_0)}{r_0+r}},\quad x\in S_{r_0},y\in S_{r},
\end{aligned}$$
where $k_{m}=\min\{\Rp k_1,\Rp k_2,\Rp k_3\}$ (see \eqref{k1k2} for the definitions of $k_1,k_2,k_3$).
From the definition of PML extensions $\wid{E}(\bs{f})$ (see \eqref{E1E2_PML}), one then obtains 
the pointwise estimate
\begin{equation}\label{keypmlextensionpestimate}
\begin{aligned}
&|\wid{E}(\bs{f})|, |\nabla \wid{E}(\bs{f})|
\le C\frac{\wid{F}(r_0,r,\sig_0)}{\left[\frac{r\sig_m(r)(r-r_0)}{r_0+r}\right]^{m_0}}
e^{-k_{m}\frac{r\sig_m(r)(r-r_0)}{r_0+r}} \|\bs{f}\|_{H^{1/2}(S_{r_0})^4},
\end{aligned}
\end{equation}
where $m_0$ is a positive integer and
$\wid{F}(r_0,r,\sig_0)$ is defined by a polynomial $\wid{P}$
\begin{equation*}
\wid{F}(r_0,r,\sig_0)=\wid{P}\left(\sqrt{(r+r_0)^2+r^2\sig_0^2},r\sqrt{1+\sig_0^2},r_0,
\sqrt{1+[\sig_0+r\sig_{m}^{'}(r)]}\right).
\end{equation*}
The following lemma on the decay property of the PML extension $\wid{E}(\bs{f})$ will play an 
important role in the convergence analysis of the PML method.
\begin{lemma}\label{poroelasticPMLpotential}
For $\bs{f}\in H^{1/2}(S_{r_0})^4$, under the condition \eqref{wavenumbercondition} and for 
sufficiently small $\omega$, we have the following
estimate for the PML extension $\wid{E}(\bs{f})=(\wid{E}_1(\bs{f})^{\top},\wid{E}_2(\bs{f})^{\top})^{\top}$
\begin{equation}\label{poropmlfse}
\Vert \wid{E}(\bs{f})\Vert_{ H^{1/2}(S_{r_2})^6}
\leq C\frac{F(r_0,r,\sig_0,d)}{(\sig_0d)^{m}}
e^{-\frac{k_m\sig_0d}{8}}\Vert \bs{f} \Vert_{ H^{1/2}(S_{r_0})^4},
\end{equation} 
where $m$ is a positive integer and $F(r_0,\sig_0,d)$ is defined by a polynomial $P$
\begin{equation*}
F(r_0,\sig_0,d)=P\left(\sqrt{(2r_0+d)^2+(r_0+d)^2\sig_0^2},\sqrt{1+\sig_0^2},r_0,d,
\sqrt{1+[\sig_0+(r_0+d)\max\limits_{r\in B_{r_2}}\sig_{m}^{'}(r)]}\right).
\end{equation*}
\begin{proof}
Choose a smooth cut-off function $\chi$ satisfying $0\leq \chi(x)\leq 1$, $\chi(x)=0$ 
for $\lvert x \rvert <r_0+d/2$ and $\chi(x)=1$ for $\lvert x \rvert >r_1-d/4$. 
For example, an explicit construction is
\begin{equation*}
\chi(r) =
\begin{cases}
0, &\ds r\le r_0+\frac{d}{2},\\
\dfrac{h\Bigl(\frac{r-(r_0+d/2)}{d/4}\Bigr)}
{h\Bigl(\frac{r-(r_0+d/2)}{d/4}\Bigr)
+ k\,h\Bigl(\frac{(r_1-d/4)-r}{d/4}\Bigr)}, &\ds r_0+\frac{d}{2}<r<r_1-\frac{d}{4},\\
1, & r\ge\ds r_1-\frac{d}{4},
\end{cases}
\end{equation*}
with $h(\tau)=e^{-1/\tau}$ for $\tau>0$ and $h(\tau)=0$ for $\tau\le0$, and $k>0$ such that
$|\chi'(r)|\le C/d$. Consequently,
\begin{equation*}
\Bigl|\frac{x}{|x|}\cdot\nabla\chi(x)\Bigr|
= |\chi'(r)| \le \frac{C}{r}\quad\text{for }r_0+d/2<r=|x|<r_0+d.
\end{equation*}
This property guarantees that all terms involving $\nabla\chi$ remain polynomially bounded.
By trace theorem,
\begin{equation}\label{Ef_estimate}
\Vert \wid{E}(\bs{f})\Vert_{ H^{1/2}(S_{r_2})^6}\leq C \Vert\chi \wid{E}(\bs{f})
\Vert_{ H^{1}(\Om_2)^6}=C\Vert\chi \wid{E}(\bs{f})\Vert_{H^1\left(\Om_2\ba\ov{B_{r_0+d/2}}\right)^6}.
\end{equation}
For a generic component $v$ of $\wid{E}(\bs{f})$ in $D=\Om_2\ba\ov{B_{r_0+d/2}}$ we have
\begin{equation}\label{H1normestimate}
\|\chi v\|_{H^1(D)}^2
= \|\chi v\|_{L^2(D)}^2 + \|\nabla(\chi v)\|_{L^2(D)^3}^2
\le \|\chi v\|_{L^2(D)}^2 + 2\bigl(\|v\nabla\chi\|_{L^2(D)^3}^2 + \|\chi\nabla v\|_{L^2(D)^3}^2\bigr).
\end{equation}
The rightmost term of the  estimate \eqref{H1normestimate} can be expressed explicitly in Spherical coordinate
$$
\begin{aligned}
&\|v\nabla\chi\|_{L^2(D)^3}^2 + \|\chi\nabla v\|_{L^2(D)^3}^2\\
=&\int_{r_0+\frac{d}{2}}^{r_0+d}\int_{0}^{\pi}\int_{0}^{2 \pi}
\Bigg[\Big\vert v\chi'(r)\Big\vert^2+\Bigg(\Big\lvert \frac{\pa v }{\pa r}\chi\Big\rvert^2
+\Big\lvert\frac{1}{r} \frac{\pa v }{\pa\theta}\chi\Big\rvert^2+\Big\lvert \frac{1}{r\sin\theta}
\frac{\pa v }{\pa \phi}\chi\Big\rvert^2\Bigg)\Bigg]r^2\sin\theta d\phi d\theta dr.
\end{aligned}$$
For $r\in (r_0+d/2,r_0+d)$,  from the definition of $\sig_m$ (see \eqref{sigma_m}), we know
$\sig_m$ is increasing and $\sig_m(r_0+d/2)\ge \sig_0/2$. Hence
\begin{equation*}
\frac{r\sig_m(r)(r-r_0)}{r_0+r}
\ge \frac{(r_0+d/2)(\sig_0/2)(d/2)}{2r_0+d} = \frac{\sig_0d}{8}.
\end{equation*}
This, combined with \eqref{keypmlextensionpestimate} yields
\begin{equation}\label{v_nablav}
|v(x)|,\; |\nabla v(x)|
\le C\frac{\wid{F}(r_0,r,\sig_0)}{\left[\frac{r\sig_m(r)(r-r_0)}{r_0+r}\right]^{m_0}}
e^{-\frac{k_m\sig_0d}{8}}
\|\bs{f}\|_{H^{1/2}(S_{r_0})^4},
\end{equation}
where $\wid{F}(r_0,r,\sig_0)$ denotes a polynomial with positive coefficients.
Combining \eqref{Ef_estimate}-\eqref{v_nablav} yields the desired estimate \eqref{poropmlfse}. 
The proof is thus complete.
\end{proof}
\end{lemma}
Given data $\bs{f}=(\bs{f}_s^{\top},f_4)^{\top}\in H^{1/2}(S_{r_0})^4$ with 
$\bs{f}_s=(f_1,f_2,f_3)^{\top}\in H^{1/2}(S_{r_0})^3$, consider the PML layer system
\begin{equation}\label{poromodifiedPMLlayer}
\begin{cases}
\omega^2(\rho\wid{\mathbf{u}} +\rho_f \wid{\mathbf{w}})+\wid{\De_c^*}\wid{\mathbf{u}} 
+\alpha M\wid{\nabla} \wid{\nabla}\cdot  \wid{\mathbf{w}}=0\quad {\rm in}\;\;\Om_{\PML},\\
\omega^2(\rho_f \wid{\mathbf{u}}+m \wid{\mathbf{w}})+{\bf{i}}\omega b \wid{\mathbf{w}}+
\wid{\nabla} \wid{\nabla} \cdot [\alpha M \wid{\mathbf{u}}+ M \wid{\mathbf{w}}]
=0 \quad {\rm in}\;\;\Om_{\PML},\\
(\wid{\mathbf{u}}^{\top},	\bs{n}\cdot\wid{\mathbf{w}})^{\top}
=(\bs{f}_s ^{\top},f_4)^{\top}\quad {\rm on} \;\;S_{r_0},\\
(\wid{\mathbf{u}}^{\top},	\bs{n}\cdot\wid{\mathbf{w}})^{\top}
=0\quad {\rm on} \;\;S_{r_2}.
\end{cases}
\end{equation} 
The well-posedness of the system \eqref{poromodifiedPMLlayer} can be similarly obtained 
following Theorem \ref{pmllayer}.
We can define the associated Dirichlet-to-Neumann operator 
$\hat{\mathcal{N}}:H^{1/2}(S_{r_0})^4\to H^{-1/2}(S_{r_0})^4$:
$$ \hat{\mathcal{N}}\bs{f}=\begin{bmatrix}
\mathcal{T}_c & \alpha M \bs{n} \nabla^{\top} \\
\alpha M \nabla^{\top} &  M \nabla^{\top}
\end{bmatrix} (\wid{\mathbf{u}}^{\top},\wid{\mathbf{w}}^{\top})^{\top}. $$
Based on Lemma \ref{poroelasticPMLpotential}, we have the error estimate between DtN operators
$\mathcal{N}$ and $\hat{\mathcal{N}}$.
\begin{theorem}
Given $\bs{f}=(\bs{f}_s^{\top},f_4)^{\top}\in H^{1/2}(S_{r_0})^4$ with $\bs{f}_s=(f_1,f_2,f_3)^{\top}$, we have
\begin{equation}\label{porodtn_error}
\Vert (\mathcal{N}-\hat{\mathcal{N}})\bs{f}\Vert_{{ H^{-1/2}( S_{r_0})}^4}
\leq C\Vert \wid{E}(\bs{f})\Vert_{{ H^{1/2}( S_{r_2})}^6} .
\end{equation}
\begin{proof}
Let $\wid{E}(\bs{f})=(\wid{E}_1(\bs{f})^{\top},\wid{E}_2(\bs{f})^{\top})^{\top}$ 
be the PML extension potential, and $\wid{V}$ be the solution of the system 
\eqref{poromodifiedPMLlayer}.
Then 
$$(\mathcal{N}-\hat{\mathcal{N}})\bs{f}=\begin{bmatrix}
\mathcal{T}_c & \alpha M \bs{n} \nabla^{\top} \\
\alpha M \nabla^{\top} &  M \nabla^{\top}
\end{bmatrix}\wid{W}|_{S_{r_0}},$$
where $\wid{W}=\wid{E}(\bs{f})-\wid{V}$ satisfies the system
\begin{equation}\label{poromodifiedPMLlayeroutb}
\begin{cases}
\omega^2(\rho\wid{\mathbf{u}} +\rho_f \wid{\mathbf{w}})+\wid{\De_c^*}\wid{\mathbf{u}} 
+\alpha M\wid{\nabla} \wid{\nabla}\cdot  \wid{\mathbf{w}}=0\quad {\rm in}\;\;\Om_{\PML},\\
\omega^2(\rho_f \wid{\mathbf{u}}+m \wid{\mathbf{w}})+{\bf{i}}\omega b \wid{\mathbf{w}}
+\wid{\nabla} \wid{\nabla} \cdot [\alpha M \wid{\mathbf{u}}+ M \wid{\mathbf{w}}]
=0 \quad {\rm in}\;\;\Om_{\PML},\\
(\wid{\mathbf{u}}^{\top},{\bs n}\cdot \wid{\mathbf{w}})^{\top}=0\quad {\rm on} \;\;S_{r_0},\\
(\wid{\mathbf{u}}^{\top},{\bs n}\cdot \wid{\mathbf{w}})^{\top} =(\wid{E}_1(\bs{f})^{\top},
{\bs n}\cdot\wid{E}_2(\bs{f})^{\top})^{\top}\quad {\rm on} \;\;S_{r_2}.
\end{cases}
\end{equation} 
By Theorem \ref{pmllayer}, the following stability estimate 
\begin{equation}\label{W_stability}
\Vert \wid{W} \Vert_{H^1(\Om_{\PML})^6}\leq C(d) \Vert (\wid{E}_1(\bs{f})^{\top},{\bs n}
\cdot\wid{E}_2(\bs{f})^{\top})^{\top}\Vert_{{ H^{1/2}(S_{r_2})}^4} \leq C(d) 
\Vert \wid{E}(\bs{f})\Vert_{{ H^{1/2}(S_{r_2})}^6}
\end{equation}
therefore holds, where the constant $C(d)$ depends on the thickness $d$ and grows at most polynomially in $d$.

Let $\Psi \in H^{1/2}( S_{r_0})^3$ be extended to a function in $H^{1}(\Om_{\PML})^{3}$ 
(still denoted by $\Psi$) such that $\Psi = 0$ on $S_{r_2}$ and $\Vert \Psi\Vert_{{ H^{1/2}( S_{r_0})}^3}
\geq \Vert\Psi \Vert_{{ H^{1}( \Om_{\PML})}^3}$. Let $g\in H^{1/2}( S_{r_0})$ be extended to a function  
$\Phi\in H^{1}(\Om_{\PML})^3$  such that $\Phi=0$ on the $ S_{r_2}$, ${\bs n}\cdot \Phi=g$ on the 
$ S_{r_0}$ and $C\Vert g\Vert_{{ H^{1/2}( S_{r_0})}}\geq \Vert\Phi \Vert_{{ H^{1}( \Om_{\PML})}^3}$.
Then it follows from the integration by parts and \eqref{W_stability} that
\begin{equation}
\begin{aligned}
\lvert \langle(\hat{\mathcal{N}}-\mathcal{N})\bs{f},(\Psi^{\top},g)^{\top}\rangle_{ S_{r_0}}\rvert
=&\lvert \mathcal{B}_{\PML}(\wid{W},(\Psi^{\top},\Phi^{\top})^{\top})\rvert\\
\leq &C \Vert \wid{W}\Vert_{H^1(\Om_{\PML})^6}\Vert (\Psi^{\top},\Phi^{\top})^{\top} \Vert_{H^1(\Om_{\PML})^6}\\
\leq& C(d)\Vert \wid{E}(\bs{f})\Vert_{{ H^{1/2}( S_{r_2})}^6}\Vert(\Psi^{\top},g)^{\top}\Vert_{{ H^{1/2}( S_{r_0})}^4},
\end{aligned} 
\end{equation}
where the sesquilinear form $\mathcal{B}_{\PML}(\cdot,\cdot)$ is defined in \eqref{poropmlbiform_B} 
for $D=\Om_{\PML}$. The it follows from the definition of $H^{-1/2}(S_{r_0})^4$ that 
\begin{equation*}
\begin{aligned}
&\Vert (\mathcal{N}-\hat{\mathcal{N}})\bs{f}\Vert_{{ H^{-1/2}( S_{r_0})}^4}\\
=& \sup\limits_{(\Psi^{\top},g)^{\top} \in H^{1/2}(\pa B_{0})^{4}\ba\{0\}}\frac{\lvert
\langle(\mathcal{N}-\hat{\mathcal{N}})\bs{f},(\Psi^{\top},g)^{\top}\rangle_{S_{r_0}}\rvert}
{\Vert (\Psi^{\top},g)^{\top}\Vert_{{ H^{1/2}(S_{r_0})}^4}}
\leq C(d)\Vert \wid{E}(\bs{f})\Vert_{{ H^{1/2}( S_{r_2})}^6}.
\end{aligned}
\end{equation*}
This, combined with Lemma \ref{poroelasticPMLpotential} arrive at the desired estimate \eqref{porodtn_error}.
\end{proof}

\end{theorem} 
Now we can establish the main result on the exponential convergence of the PML method.
\begin{theorem}[Convergence Theorem] For given 
${\bf Q} \in H^1(\Om^c)^3$ with compact support inside $B_{r_0}$, let $(\mathbf{u}^{\top},
\mathbf{w}^{\top})^{\top}\in V_{\pa\Om}(\Om_0)$ and $(\wid{\mathbf{u}}^{\top},
\wid{\mathbf{w}}^{\top})^{\top}\in H_0^1(\Om_2)^3\times V(\Om_2)$ be the solutions of the 
problem \eqref{sporoelastic1_1}-\eqref{sDBC2_1} and  \eqref{poromodifiedPML}, respectively. 
Then under the condition \eqref{wavenumbercondition} and the smallness assumption on $\alpha$ 
and $\omega$, the following error estimate holds:
\begin{equation}\label{porosolution_err}
\Vert (\wid{\mathbf{u}}^{\top},\wid{\mathbf{w}}^{\top})^{\top}-(\mathbf{u}^{\top},
\mathbf{w}^{\top})^{\top} \Vert_{H^1(\Om_0)^6} \leq  C\frac{F(r_0,\sig_0,d)}{(\sig_0d)^{m}}
e^{-\frac{k_m\sig_0d}{8}}\Vert (\wid{\mathbf{u}}^{\top},\bs{n}\cdot\wid{\mathbf{w}})^{\top} 
\Vert_{ H^{1/2}(S_{r_0})^4},
\end{equation}
where $m$ is a positive integer and $F(r_0,\sig_0,d)$ is defined by a polynomial $P$
\begin{equation*}
F(r_0,\sig_0,d)=P\left(\sqrt{(2r_0+d)^2+(r_0+d)^2\sig_0^2},\sqrt{1+\sig_0^2},r_0,d,
\sqrt{1+[\sig_0+(r_0+d)\max\limits_{r\in B_{r_2}}\sig_{m}^{'}(r)]}\right).
\end{equation*}
\begin{proof} 
Noting that ${\bf Q} \in H^1(\Om^c)^3$ is compactly supported inside $B_{r_0}$, then 
$(\mathbf{u}^{\top},\mathbf{w}^{\top})^{\top}$ and 
$(\wid{\mathbf{u}}^{\top},\wid{\mathbf{w}}^{\top})^{\top}$ satisfy the following variational problems
\begin{align}\label{scattering}
\mathcal{B}((\mathbf{u}^{\top},\mathbf{w}^{\top})^{\top},\Psi)&=-\int_{\Om_0}{\bf Q}\cdot \ov{\Psi}dx,
\quad \forall\;\Psi \in V_{\pa \Om}(\Om_0), \\ \label{pmlproblem}
\mathcal{B}_{\Om_2}((\wid{\mathbf{u}}^{\top},\wid{\mathbf{w}}^{\top})^{\top},\Psi)&
=-\int_{\Om_0}{\bf Q}\cdot \ov{\Psi}dx,
\quad \forall\;\Psi=(\Psi^s,\Psi^f)\in H_0^1(\Om_2)^3\times V(\Om_2).
\end{align}
Simple decomposition of $\mathcal{B}_{\Om_2}$ and integration by parts imply
\begin{equation}\label{B_decom}
\begin{aligned}
\mathcal{B}_{\Om_2}((\wid{\mathbf{u}}^{\top},\wid{\mathbf{w}}^{\top})^{\top},\Psi)
&= \mathcal{B}((\wid{\mathbf{u}}^{\top},\wid{\mathbf{w}}^{\top})^{\top},\Psi)+
\langle\mathcal{N}\wid{U},\Psi_{\bs{n}}\rangle_{S_{r_0}}+\mathcal{B}_{\Om_{\PML}}
((\wid{\mathbf{u}}^{\top},\wid{\mathbf{w}}^{\top})^{\top},\Psi)\\&
=\mathcal{B}((\wid{\mathbf{u}}^{\top},\wid{\mathbf{w}}^{\top})^{\top},\Psi)
+\langle\mathcal{N}\wid{U},\Psi_{\bs{n}}\rangle_{S_{r_0}}-
\langle\hat{\mathcal{N}}\wid{U},\Psi_{\bs{n}}\rangle_{S_{r_0}},
\end{aligned} 
\end{equation}
where $$ \wid{U}=\begin{bmatrix}
\wid{\mathbf{u}} \\
\bs{n}\cdot \wid{\mathbf{w}} 	
\end{bmatrix}\in H^{1/2}(S_{r_0})^4,\quad \Psi_{\bs{n}}=\begin{bmatrix}
\Psi^s \\
\bs{n}\cdot \Psi^f 
\end{bmatrix}\in H^{1/2}(S_{r_0})^4. $$
Subtracting the equations \eqref{scattering} and \eqref{pmlproblem} and using \eqref{B_decom} gives
\begin{equation}
\begin{aligned}
\mathcal{B}((\mathbf{u}^{\top},\mathbf{w}^{\top})^{\top}-(\wid{\mathbf{u}}^{\top},
\wid{\mathbf{w}}^{\top})^{\top},\Psi)=\langle (\mathcal{N}-\hat{\mathcal{N}})\wid{U},
\Psi_{\bs{n}}\rangle_{S_{r_0}}
\end{aligned} 
\end{equation}
This, combined with the trace theorem gives
\begin{equation}
\begin{aligned}
\lvert \mathcal{B}((\mathbf{u}^{\top},\mathbf{w}^{\top})^{\top}-(\wid{\mathbf{u}}^{\top},
\wid{\mathbf{w}}^{\top})^{\top},\Psi)\rvert&\leq C\Vert (\mathcal{N}-\hat{\mathcal{N}})
\wid{U}\Vert_{ H^{-1/2}(S_{r_0})^4}\Vert\Psi _{\bs{n}}\Vert_{H^{1/2}(S_{r_0})^4}.
\end{aligned} 
\end{equation}
which obviously implies
\begin{equation}\label{eq:label}
\frac{\lvert \mathcal{B}((\mathbf{u}^{\top},\mathbf{w}^{\top})^{\top}-
(\wid{\mathbf{u}}^{\top},\wid{\mathbf{w}}^{\top})^{\top},\Psi)\rvert}{\Vert\Psi \Vert_{H^{1}(\Om_0)^6}}
\leq C\Vert (\mathcal{N}-\hat{\mathcal{N}})\wid{U}\Vert_{ H^{-1/2}(S_{r_0})^4}.
\end{equation}
The desired error estimate \eqref{porosolution_err} follows from Lemma \ref{porodtn_error}
and the inf-sup condition \eqref{infsup_B1} for the sesquilinear form $ \mathcal{B}$.
\end{proof} 
\end{theorem}

\section{Conclusion}\label{sec:conclusion}

In this paper, we have studied three-dimensional time-harmonic poroelastic scattering problems.
We first introduced an equivalent $\mathbf{u}-p$ system with fewer degrees of freedom, derived 
the associated fundamental solution and Green's identity, and proved that under reasonable parameter 
constraints the real and imaginary parts of all three wave numbers are positive. Under certain assumptions,
 we proved the well-posedness of the PML problems in the truncated domain and in the layer.
Moreover, the exponential convergence of the spherical PML method was established in terms of the 
layer thickness and the absorbing parameters. The proof relies on an error estimate between the 
Dirichlet-to-Neumann (DtN) operators of the original scattering problem and the truncated PML problem, 
which in turn follows from the exponential decay of the PML extensions of the stretched fundamental solution.

Our approach can be extended to other poroelastic scattering models and to more general boundary conditions 
on the truncated PML boundary, such as the mixed boundary conditions. It would also be of interest to study 
the uniaxial PML method for time-harmonic poroelastic scattering problems. We hope to report such results in the future.

\section*{Acknowledgements}

This work is partly supported by the Fundamental Research Funds for the Central Universities 2026JBMC043, 
and the NNSF of China grants 12201033 and 12431016.

\begin{appendices}
\section{The fundamental solution of an abstract Biot system}
\label{app:aBiot_fs}
The matrix form of the abstract Biot's system is given by
\begin{equation*}
B(\mathbf{u}^{\top},p)^{\top}=0,\quad{\rm where}\;B=\begin{bmatrix}
\Delta^*+T_1\mat{I}_3 & T_2\nabla \\
T_3\nabla^{\top} & \Delta +T_4\\
\end{bmatrix}.
\end{equation*}
By direct calculation, the elements of the adjoint matrix of $B$ are:
\begin{equation*}
\begin{aligned} 
{\rm adj}B_{11}= &\;(T_1+\mu\Delta)\left[T_1T_4+[T_4(\la+2\mu)-T_2T_3]\left(\frac{\pa ^2}
{\pa x_2^2}+\frac{\pa ^2}{\pa x_3^2}\right)+T_4\mu\frac{\pa ^2}{\pa x_1^2}\right.\\
&\left.\qquad\qquad\quad+\mu \frac{\pa ^2}{\pa x_1^2}\Delta +(\la+2\mu)\left(\frac{\pa ^2}
{\pa x_2^2}+\frac{\pa ^2}{\pa x_3^2}\right)\Delta+T_1\Delta\right],\\
{\rm adj}B_{12}=&\;(T_1+\mu\Delta) \frac{\pa ^2}{\pa x_1\pa x_2}[T_2T_3-(\la +\mu)T_4-(\la +\mu)\Delta],\\
{\rm adj}B_{13}=&\;(T_1+\mu\Delta) \frac{\pa ^2}{\pa x_1\pa x_3}[T_2T_3-(\la +\mu)T_4-(\la +\mu)\Delta],\\
{\rm adj}B_{14}=&\;-T_2 \frac{\pa }{\pa x_1}(T_1+\mu\Delta)^2,\\
{\rm adj}B_{21}=&\;(T_1+\mu\Delta) \frac{\pa ^2}{\pa x_2\pa x_1}[T_2T_3-(\la +\mu)T_4-(\la +\mu)\Delta],\\
\end{aligned}
\end{equation*}
\begin{equation*}
\begin{aligned}
{\rm adj}B_{22}= &\;(T_1+\mu\Delta)\left[T_1T_4+[T_4(\la+2\mu)-T_2T_3]\left(\frac{\pa ^2}
{\pa x_1^2}+\frac{\pa ^2}{\pa x_3^2}\right)+T_4\mu\frac{\pa ^2}{\pa x_2^2}\right.\\
&\left.\qquad\qquad\quad+\mu \frac{\pa ^2}{\pa x_2^2}\Delta +(\la+2\mu)\left(\frac{\pa ^2}
{\pa x_1^2}+\frac{\pa ^2}{\pa x_3^2}\right)\Delta+T_1\Delta\right],\\
{\rm adj}B_{23}=&\;(T_1+\mu\Delta) \frac{\pa ^2}{\pa x_2\pa x_3}[T_2T_3-(\la +\mu)T_4-(\la +\mu)\Delta],\\
{\rm adj}B_{24}=&\;-T_2\frac{\pa }{\pa x_2}(T_1+\mu\Delta)^2, \\
{\rm adj}B_{31}=&\;(T_1+\mu\Delta) \frac{\pa ^2}{\pa x_3\pa x_1}[T_2T_3-(\la +\mu)T_4-(\la +\mu)\Delta],\\
{\rm adj}B_{32}=&\;(T_1+\mu\Delta) \frac{\pa ^2}{\pa x_3\pa x_2}[T_2T_3-(\la +\mu)T_4-(\la +\mu)\Delta],\\
{\rm adj}B_{33}= &\;(T_1+\mu\Delta)\left[T_1T_4+[T_4(\la+2\mu)-T_2T_3]\left(\frac{\pa ^2}{\pa x_1^2}
+\frac{\pa ^2}{\pa x_2^2}\right)+T_4\mu\frac{\pa ^2}{\pa x_3^2}\right.\\
&\left.\qquad\qquad\quad+\mu \frac{\pa ^2}{\pa x_3^2}\Delta +(\la+2\mu)\left(\frac{\pa ^2}{\pa x_1^2}
+\frac{\pa ^2}{\pa x_2^2}\right)\Delta+T_1\Delta\right],\\
{\rm adj}B_{34}=&\;-T_2\frac{\pa }{\pa x_3}(T_1+\mu\Delta)^2, \quad
{\rm adj}B_{41}=\;-T_3 \frac{\pa }{\pa x_1}(T_1+\mu\Delta)^2,\\
{\rm adj}B_{42}=&\;-T_3 \frac{\pa }{\pa x_2}(T_1+\mu\Delta)^2,\quad
{\rm adj}B_{43}=\;-T_3\frac{\pa }{\pa x_3}(T_1+\mu\Delta)^2,\\
{\rm adj}B_{44}=&\;(T_1+\mu\Delta)^2 [T_1+(\la+2\mu)\Delta].
\end{aligned}
\end{equation*}
The elements of the fundamental solution matrix for the above abstract Biot system  are 
(with the singularity at the origin):
\begin{equation*}
\begin{aligned} 
\Phi_{11}(x)
=&\;\frac{1}{\mu(\la+2\mu)}\left\{\left[T_1T_4+[T_4(\la+2\mu)-T_2T_3]\left(\frac{\pa ^2}{\pa x_2^2}
+\frac{\pa ^2}{\pa x_3^2}\right)+T_4\mu\frac{\pa ^2}{\pa x_1^2}\right]\right.\\
&\left(A_1\frac{e^{{\bf i}k_1\lvert x \rvert}}{4\pi\lvert x \rvert}+A_2\frac{e^{{\bf i}k_2\lvert x 
\rvert}}{4\pi\lvert x \rvert}+A_3\frac{e^{{\bf i}k_3\lvert x \rvert}}{4\pi\lvert x \rvert}\right)
-\Bigg[\mu \frac{\pa ^2}{\pa x_1^2} +(\la+2\mu)\left(\frac{\pa ^2}{\pa x_2^2}+\frac{\pa ^2}
{\pa x_3^2}\right)+T_1\Bigg]\\
&\left.\left(k_1^2A_1\frac{e^{{\bf i}k_1\lvert x \rvert}}{4\pi\lvert x \rvert}+k_2^2A_2
\frac{e^{{\bf i}k_2\lvert x \rvert}}{4\pi\lvert x \rvert}+k_3^2A_3\frac{e^{{\bf i}k_3\lvert x \rvert}}
{4\pi\lvert x \rvert}\right)\right\},\\
\Phi_{12}(x)
=&\;\frac{T_2T_3-(\la +\mu)T_4}{\mu(\la+2\mu)}\frac{\pa ^2}{\pa x_1\pa x_2}\left(A_1
\frac{e^{{\bf i}k_1\lvert x \rvert}}{4\pi\lvert x \rvert}+A_2\frac{e^{{\bf i}k_2
\lvert x \rvert}}{4\pi\lvert x \rvert}+A_3\frac{e^{{\bf i}k_3\lvert x \rvert}}
{4\pi\lvert x \rvert}\right)\\
&+   \frac{\la +\mu}{\mu(\la+2\mu)}\frac{\pa ^2}{\pa x_1\pa x_2}\left(k_1^2A_1\frac{e^{{\bf i}
k_1\lvert x \rvert}}{4\pi\lvert x \rvert}+k_2^2A_2\frac{e^{{\bf i}k_2\lvert x \rvert}}
{4\pi\lvert x \rvert}+k_3^2A_3\frac{e^{{\bf i}k_3\lvert x \rvert}}{4\pi\lvert x \rvert}\right),\\
\Phi_{13}(x)
=&\;\frac{T_2T_3-(\la +\mu)T_4}{\mu(\la+2\mu)} \frac{\pa ^2}{\pa x_1\pa x_3}
\left(A_1\frac{e^{{\bf i}k_1\lvert x \rvert}}{4\pi\lvert x \rvert}+A_2\frac{e^{{\bf i}
k_2\lvert x \rvert}}{4\pi\lvert x \rvert}+A_3\frac{e^{{\bf i}k_3\lvert x \rvert}}
{4\pi\lvert x \rvert}\right)\\
&+  \frac{\la +\mu}{\mu(\la+2\mu)} \frac{\pa ^2}{\pa x_1\pa x_3}\left(k_1^2A_1
\frac{e^{{\bf i}k_1\lvert x \rvert}}{4\pi\lvert x \rvert}+k_2^2A_2\frac{e^{{\bf i}
k_2\lvert x \rvert}}{4\pi\lvert x \rvert}+k_3^2A_3\frac{e^{{\bf i}k_3\lvert x \rvert}}
{4\pi\lvert x \rvert}\right),
\end{aligned}
\end{equation*}
\begin{equation*}
\begin{aligned} 
\Phi_{14}(x)
=&\;-\frac{T_1T_2}{\mu(\la+2\mu)}  \frac{\pa }{\pa x_1}\left(A_1\frac{e^{{\bf i}k_1
\lvert x \rvert}}{4\pi\lvert x \rvert}+A_2\frac{e^{{\bf i}k_2\lvert x \rvert}}
{4\pi\lvert x \rvert}+A_3\frac{e^{{\bf i}k_3\lvert x \rvert}}{4\pi\lvert x \rvert}\right)\\
&+\frac{T_2}{\la+2\mu} \frac{\pa }{\pa x_1}\left(k_1^2A_1\frac{e^{{\bf i}k_1\lvert x 
\rvert}}{4\pi\lvert x \rvert}+k_2^2A_2\frac{e^{{\bf i}k_2\lvert x \rvert}}{4\pi\lvert x 
\rvert}+k_3^2A_3\frac{e^{{\bf i}k_3\lvert x \rvert}}{4\pi\lvert x \rvert}\right),\\
\Phi_{21}(x)
=&\;\frac{T_2T_3-(\la +\mu)T_4}{\mu(\la+2\mu)} \frac{\pa ^2}{\pa x_2\pa x_1}\left(A_1
\frac{e^{{\bf i}k_1\lvert x \rvert}}{4\pi\lvert x \rvert}+A_2\frac{e^{{\bf i}k_2\lvert 
x \rvert}}{4\pi\lvert x \rvert}+A_3\frac{e^{{\bf i}k_3\lvert x \rvert}}{4\pi\lvert x \rvert}\right)\\
&+\frac{\la +\mu}{\mu(\la+2\mu)} \frac{\pa ^2}{\pa x_2\pa x_1}\left(k_1^2A_1\frac{e^{{\bf i}
k_1\lvert x \rvert}}{4\pi\lvert x \rvert}+k_2^2A_2\frac{e^{{\bf i}k_2\lvert x \rvert}}
{4\pi\lvert x \rvert}+k_3^2A_3\frac{e^{{\bf i}k_3\lvert x \rvert}}{4\pi\lvert x \rvert}\right),\\
\Phi_{22}(x)=&\;\frac{1}{\mu(\la+2\mu)}\left\{\left[T_1T_4+[T_4(\la+2\mu)-T_2T_3]
\left(\frac{\pa ^2}{\pa x_1^2}+\frac{\pa ^2}{\pa x_3^2}\right)+T_4\mu\frac{\pa ^2}{\pa x_2^2}\right]\right.\\
&\left(A_1\frac{e^{{\bf i}k_1\lvert x \rvert}}{4\pi\lvert x \rvert}+A_2\frac{e^{{\bf i}k_2
\lvert x \rvert}}{4\pi\lvert x \rvert}+A_3\frac{e^{{\bf i}k_3\lvert x \rvert}}{4\pi\lvert x
 \rvert}\right)-\Bigg[\mu \frac{\pa ^2}{\pa x_2^2} +(\la+2\mu)\left(\frac{\pa ^2}{\pa x_1^2}
 +\frac{\pa ^2}{\pa x_3^2}\right)+T_1\Bigg]\\
&\left.\left(k_1^2A_1\frac{e^{{\bf i}k_1\lvert x \rvert}}{4\pi\lvert x \rvert}+k_2^2A_2
\frac{e^{{\bf i}k_2\lvert x \rvert}}{4\pi\lvert x \rvert}+k_3^2A_3\frac{e^{{\bf i}k_3\lvert 
x \rvert}}{4\pi\lvert x \rvert}\right)\right\},\\
\Phi_{23}(x)
=&\;\frac{T_2T_3-(\la +\mu)T_4}{\mu(\la+2\mu)} \frac{\pa ^2}{\pa x_2\pa x_3}\left(A_1
\frac{e^{{\bf i}k_1\lvert x \rvert}}{4\pi\lvert x \rvert}+A_2\frac{e^{{\bf i}k_2\lvert x \rvert}}{4\pi\lvert x \rvert}+A_3\frac{e^{{\bf i}k_3\lvert x \rvert}}{4\pi\lvert x \rvert}\right)\\
&+\frac{\la +\mu}{\mu(\la+2\mu)} \frac{\pa ^2}{\pa x_2\pa x_3}\left(k_1^2A_1
\frac{e^{{\bf i}k_1\lvert x \rvert}}{4\pi\lvert x \rvert}+k_2^2A_2\frac{e^{{\bf i}
k_2\lvert x \rvert}}{4\pi\lvert x \rvert}+k_3^2A_3\frac{e^{{\bf i}k_3\lvert x \rvert}}
{4\pi\lvert x \rvert}\right),\\
\Phi_{24}(x)
=&\;-\frac{T_1T_2}{\mu(\la+2\mu)}  \frac{\pa }{\pa x_2}\left(A_1\frac{e^{{\bf i}k_1\lvert x 
\rvert}}{4\pi\lvert x \rvert}+A_2\frac{e^{{\bf i}k_2\lvert x \rvert}}{4\pi\lvert x \rvert}+
A_3\frac{e^{{\bf i}k_3\lvert x \rvert}}{4\pi\lvert x \rvert}\right)\\
&+\frac{T_2}{\la+2\mu} \frac{\pa }{\pa x_2}\left(k_1^2A_1\frac{e^{{\bf i}k_1\lvert x \rvert}}
{4\pi\lvert x \rvert}+k_2^2A_2\frac{e^{{\bf i}k_2\lvert x \rvert}}{4\pi\lvert x \rvert}+k_3^2A_3
\frac{e^{{\bf i}k_3\lvert x \rvert}}{4\pi\lvert x \rvert}\right),\\
\Phi_{31}(x)
=&\;\frac{T_2T_3-(\la +\mu)T_4}{\mu(\la+2\mu)} \frac{\pa ^2}{\pa x_3\pa x_1}\left(A_1
\frac{e^{{\bf i}k_1\lvert x \rvert}}{4\pi\lvert x \rvert}+A_2\frac{e^{{\bf i}k_2\lvert x \rvert}}
{4\pi\lvert x \rvert}+A_3\frac{e^{{\bf i}k_3\lvert x \rvert}}{4\pi\lvert x \rvert}\right)\\
&+\frac{\la +\mu}{\mu(\la+2\mu)} \frac{\pa ^2}{\pa x_3\pa x_1}\left(k_1^2A_1\frac{e^{{\bf i}k_1
\lvert x \rvert}}{4\pi\lvert x \rvert}+k_2^2A_2\frac{e^{{\bf i}k_2\lvert x \rvert}}{4\pi\lvert x 
\rvert}+k_3^2A_3\frac{e^{{\bf i}k_3\lvert x \rvert}}{4\pi\lvert x \rvert}\right),\\
\Phi_{32}(x)
=&\;\frac{T_2T_3-(\la +\mu)T_4}{\mu(\la+2\mu)} \frac{\pa ^2}{\pa x_3\pa x_2}\left(A_1\frac{e^{{\bf i}
k_1\lvert x \rvert}}{4\pi\lvert x \rvert}+A_2\frac{e^{{\bf i}k_2\lvert x \rvert}}{4\pi\lvert x \rvert}
+A_3\frac{e^{{\bf i}k_3\lvert x \rvert}}{4\pi\lvert x \rvert}\right)\\
&+\frac{\la +\mu}{\mu(\la+2\mu)} \frac{\pa ^2}{\pa x_3\pa x_2}\left(k_1^2A_1\frac{e^{{\bf i}k_1\lvert 
x \rvert}}{4\pi\lvert x \rvert}+k_2^2A_2\frac{e^{{\bf i}k_2\lvert x \rvert}}{4\pi\lvert x \rvert}
+k_3^2A_3\frac{e^{{\bf i}k_3\lvert x \rvert}}{4\pi\lvert x \rvert}\right),\\
\Phi_{33}(x)=&\;\frac{1}{\mu(\la+2\mu)}\left\{\left[T_1T_4+[T_4(\la+2\mu)-T_2T_3]\left(\frac{\pa ^2}
{\pa x_1^2}+\frac{\pa ^2}{\pa x_2^2}\right)+T_4\mu\frac{\pa ^2}{\pa x_3^2}\right]\right.\\
&\left(A_1\frac{e^{{\bf i}k_1\lvert x \rvert}}{4\pi\lvert x \rvert}+A_2\frac{e^{{\bf i}k_2\lvert x 
\rvert}}{4\pi\lvert x \rvert}+A_3\frac{e^{{\bf i}k_3\lvert x \rvert}}{4\pi\lvert x \rvert}\right)-
\Bigg[\mu \frac{\pa ^2}{\pa x_3^2} +(\la+2\mu)\left(\frac{\pa ^2}{\pa x_1^2}+\frac{\pa ^2}
{\pa x_2^2}\right)+T_1\Bigg]\\
&\left.\left(k_1^2A_1\frac{e^{{\bf i}k_1\lvert x \rvert}}{4\pi\lvert x \rvert}+k_2^2A_2
\frac{e^{{\bf i}k_2\lvert x \rvert}}{4\pi\lvert x \rvert}+k_3^2A_3\frac{e^{{\bf i}k_3
\lvert x \rvert}}{4\pi\lvert x \rvert}\right)\right\},\\
\end{aligned}
\end{equation*}
\begin{equation*}
\begin{aligned} 
\Phi_{34}(x)
=&\;-\frac{T_1T_2}{\mu(\la+2\mu)} \frac{\pa }{\pa x_3} \left(A_1\frac{e^{{\bf i}
k_1\lvert x \rvert}}{4\pi\lvert x \rvert}+A_2\frac{e^{{\bf i}k_2\lvert x \rvert}}{4\pi
\lvert x \rvert}+A_3\frac{e^{{\bf i}k_3\lvert x \rvert}}{4\pi\lvert x \rvert}\right)\\
&+\frac{T_2}{\la+2\mu} \frac{\pa }{\pa x_3}\left(k_1^2A_1\frac{e^{{\bf i}k_1\lvert x \rvert}}
{4\pi\lvert x \rvert}+k_2^2A_2\frac{e^{{\bf i}k_2\lvert x \rvert}}{4\pi\lvert x \rvert}+k_3^2A_3
\frac{e^{{\bf i}k_3\lvert x \rvert}}{4\pi\lvert x \rvert}\right),\\
\Phi_{41}(x)
=&\;-\frac{T_1T_3}{\mu(\la+2\mu)}  \frac{\pa }{\pa x_1}\left(A_1\frac{e^{{\bf i}k_1\lvert x \rvert}}
{4\pi\lvert x \rvert}+A_2\frac{e^{{\bf i}k_2\lvert x \rvert}}{4\pi\lvert x \rvert}+A_3\frac{e^{{\bf i}
k_3\lvert x \rvert}}{4\pi\lvert x \rvert}\right)\\
&+\frac{T_3}{\la+2\mu} \frac{\pa }{\pa x_1}\left(k_1^2A_1\frac{e^{{\bf i}k_1\lvert x \rvert}}{4\pi
\lvert x \rvert}+k_2^2A_2\frac{e^{{\bf i}k_2\lvert x \rvert}}{4\pi\lvert x \rvert}+k_3^2A_3
\frac{e^{{\bf i}k_3\lvert x \rvert}}{4\pi\lvert x \rvert}\right),\\
\Phi_{42}(x)
=&\;-\frac{T_1T_3}{\mu(\la+2\mu)}  \frac{\pa }{\pa x_2}\left(A_1\frac{e^{{\bf i}k_1\lvert x \rvert}}
{4\pi\lvert x \rvert}+A_2\frac{e^{{\bf i}k_2\lvert x \rvert}}{4\pi\lvert x \rvert}+A_3\frac{e^{{\bf i}
k_3\lvert x \rvert}}{4\pi\lvert x \rvert}\right)\\
&+\frac{T_3}{\la+2\mu}\frac{\pa }{\pa x_2}\left(k_1^2A_1\frac{e^{{\bf i}k_1\lvert x \rvert}}
{4\pi\lvert x \rvert}+k_2^2A_2\frac{e^{{\bf i}k_2\lvert x \rvert}}{4\pi\lvert x \rvert}+k_3^2A_3
\frac{e^{{\bf i}k_3\lvert x \rvert}}{4\pi\lvert x \rvert}\right),\\
\Phi_{43}(x)
=&\;-\frac{T_1T_3}{\mu(\la+2\mu)}\frac{\pa }{\pa x_3}\left(A_1\frac{e^{{\bf i}k_1\lvert x \rvert}}
{4\pi\lvert x \rvert}+A_2\frac{e^{{\bf i}k_2\lvert x \rvert}}{4\pi\lvert x \rvert}+A_3\frac{e^{{\bf i}
k_3\lvert x \rvert}}{4\pi\lvert x \rvert}\right)\\
&+\frac{T_3}{\la+2\mu}\frac{\pa }{\pa x_3}\left(k_1^2A_1\frac{e^{{\bf i}k_1\lvert x \rvert}}
{4\pi\lvert x \rvert}+k_2^2A_2\frac{e^{{\bf i}k_2\lvert x \rvert}}{4\pi\lvert x \rvert}+k_3^2A_3
\frac{e^{{\bf i}k_3\lvert x \rvert}}{4\pi\lvert x \rvert}\right),\\
\Phi_{44}(x)
=&\;\frac{T_1^2}{\mu(\la+2\mu)}\left(A_1\frac{e^{{\bf i}k_1\lvert x \rvert}}{4\pi\lvert x \rvert}
+A_2\frac{e^{{\bf i}k_2\lvert x \rvert}}{4\pi\lvert x \rvert}+A_3\frac{e^{{\bf i}k_3\lvert x \rvert}}
{4\pi\lvert x \rvert}\right)\\
&+\frac{T_1(\la+3\mu)}{\mu(\la+2\mu)}\left(k_1^2A_1\frac{e^{{\bf i}k_1\lvert x \rvert}}{4\pi
\lvert x \rvert}+k_2^2A_2\frac{e^{{\bf i}k_2\lvert x \rvert}}{4\pi\lvert x \rvert}+k_3^2A_3
\frac{e^{{\bf i}k_3\lvert x \rvert}}{4\pi\lvert x \rvert}\right)\\
&+\left(k_1^4A_1\frac{e^{{\bf i}k_1\lvert x \rvert}}{4\pi\lvert x \rvert}+k_2^4A_2
\frac{e^{{\bf i}k_2\lvert x \rvert}}{4\pi\lvert x \rvert}+k_3^4A_3\frac{e^{{\bf i}k_3
\lvert x \rvert}}{4\pi\lvert x \rvert}\right).
\end{aligned}
\end{equation*}
\end{appendices}

\end {document}